%% file: ht08-crelle10.tex
\documentclass[11pt]{article}

\newif\ifproblems

\ifproblems
  \newcommand{\og}{``}
  \newcommand{\fg}{''}
\else
  \usepackage[english,francais]{babel}
  \usepackage{hyperref}
\fi
  \usepackage[leqno]{amsmath}
  \usepackage{amsthm, amssymb}
  \usepackage[all,cmtip,poly]{xy}
  \usepackage{epsfig}
  \newcommand{\color}[2][1]{}
\usepackage{amsfonts}
 
\usepackage{fancyhdr}
\theoremstyle{plain} 
 \newtheorem{lem}{Lemme}[section]
 \newtheorem{pro}[lem]{Proposition}
 \newtheorem{thm}{Th{\'e}or{\`e}me}
 \newtheorem{cor}[lem]{Corollaire}
 \newtheorem{sublem}[lem]{Sous-lemme}

 \newtheorem{conj}{Conjecture}
\theoremstyle{definition}
 \newtheorem{de}[lem]{D{\'e}finition}
 \newtheorem{rk}[lem]{Remarque}

 \newcommand\hfld[2]{\smash{\mathop{\hbox to 12mm{\rightarrowfill}}
      \limits^{\scriptstyle#1}_{\scriptstyle#2}}}
 \newcommand\hflg[2]{\smash{\mathop{\hbox to 6mm{\leftarrowfill}}
      \limits^{\scriptstyle#1}_{\scriptstyle#2}}}

 \newcommand\hfl[1]{\,\smash{\mathop{\hbox to 6mm{\rightarrowfill}}
      \limits^{\scriptstyle#1}}\,}

 \newcommand{\cK}{{\cal K}}

 \newcommand{\bF}{{\Bbb F}}

 \newcommand{\cO}{{\cal O}}

 \newcommand{\cW}{{\cal W}}

 \newcommand{\cV}{{\cal V}}

 \newcommand{\cH}{{\cal H}}

 \def\a{{\frak a}}
 \def\p{{\frak p}}
 \def\pb{\overline{\frak p}}

 \def\f{{\frak f}}
 \def\q{{\frak q}}

 \def\C{{\Bbb C}}
 \def\m{{\frak m}}
 \def\s{{\frak s}}

 \def\A{{\Bbb A}}
 
 \def\Z{{\Bbb Z}}
 \def\Zp{{{\Bbb Z}_p}}
 \def\Q{{\Bbb Q}}
 \def\R{{\Bbb R}}
 \def\C{{\Bbb C}}
 \def\F{{\Bbb F}}
 \def\T{{\Bbb T}}

 \def\G{{{\Bbb G}}}

 \def\adots{\mathinner{\mkern2mu\raise1pt\hbox{.}\mkern3mu\raise4pt\hbox{.}\mkern1mu\raise7pt\hbox{.}}}

\def\Ind{{\rm Ind}}

\DeclareMathOperator{\GL}{GL}

\DeclareMathOperator{\GSp}{GSp}
\DeclareMathOperator{\Sp}{Sp}
\DeclareMathOperator{\JH}{JH}
\newcommand{\Mor}{\operatorname{Mor}}
\newcommand{\Gal}{\operatorname{Gal}}
\newcommand{\ord}{\operatorname{ord}}
\newcommand{\Fr}{\operatorname{Fr}}
\newcommand{\gr}{\operatorname{gr}}
\newcommand{\nr}{\operatorname{nr}}
\newcommand{\diag}{\operatorname{diag}}
\newcommand{\spin}{\operatorname{spin}}
\newcommand{\fp}{\F_p}
\newcommand{\fpb}{\overline \F_p}

\newcommand{\soc}{\operatorname{soc}}
\newcommand{\sgn}{\operatorname{sgn}}
\newcommand{\simi}{{\operatorname{sim}}}

\renewcommand{\s}{^\times}
\newcommand{\dual}{^\vee}

\newcommand{\RR}{\ensuremath{\mathcal R}}
\newcommand{\qb}{\overline\Q}
\newcommand{\qp}{\Q_p}
\newcommand{\zp}{\Z_p}

\newcommand{\qpb}{\overline\Q_p}
\newcommand{\zpb}{\overline\Z_p}
\newcommand{\rhob}{\overline\rho}
\newcommand{\rhot}{\widetilde\rho}
\newcommand{\reg}{_\mathrm{r\acute{e}g}}
\newcommand{\rg}{\operatorname{rg}}
\newcommand{\ta}{\widetilde\alpha}

\newcommand{\td}{\widetilde\delta}

\newcommand{\into}{\hookrightarrow}

\newcommand{\congto}{\xrightarrow{\,\sim\,}}

\newcommand{\topspace}[1]{\vbox{\vspace*{#1}}}

\newcommand{\cbull}{\mathbin{\hspace{2pt}\begin{picture}(1,1)(0,-2.5)\circle*{2}\end{picture}\hspace{1pt}}}
\newcommand{\expbull}{{\mathop{\hspace{1.5pt}\begin{picture}(1,1)(0,-2)\circle*{2}\end{picture}\hspace{0.5pt}}}}

\renewcommand{\(}{\textup{(}}
\renewcommand{\)}{\textup{)}}

 \title{Conjecture de type de Serre\\ et formes compagnons pour $\GSp_4$}
 \author{F.\ Herzig\thanks{Partiellement soutenu par la NSF (grant DMS-0902044 et accord DMS-0635607).}\and J.\ Tilouine\thanks{Partiellement soutenu par le Projet Blanc ANR-10-BLAN 0114.}}
 \date{} %enleve pour l'instant (FH/Oct 2010)
 
\begin{document}
 \maketitle 

% change the symbol for itemize + enumerate
\renewcommand{\labelitemi}{$\bullet$}
\renewcommand{\labelenumi}{\upshape(\roman{enumi})}

\begin{abstract}
  On pr\'esente une conjecture de type de Serre sur la modularit\'e d'une repr\'esentation galoisienne modulo~$p$ de rang~$4$
  \`a valeurs dans le groupe symplectique. Nous supposons que la repr\'esentation est irr\'eductible et impaire (au sens
  symplectique).  Nous formulons la condition de modularit\'e en utilisant la cohomologie \'etale et la cohomologie de de
  Rham alg\'ebrique des vari\'et\'es de Siegel de niveau premier \`a~$p$.  On se concentre sur le cas o\`u la repr\'esentation
  est ordinaire en~$p$ et on donne une liste des poids de Serre correspondante. Lorsque la repr\'esentation est mod\'er\'ee,
  nous conjecturons que chaque poids de cette liste est r\'ealis\'e, et sinon, nous d\'ecrivons un sous-ensemble des poids de
  la liste qui doivent \^ etre r\'ealis\'es.  Nous proposons une construction des classes de cohomologie de de Rham qui
  r\'ealisent certains de ces poids \`a l'aide du complexe BGG dual.
\end{abstract}

\selectlanguage{english}

\begin{abstract}
  We present a Serre-type conjecture on the modularity of four-dimensional symplectic mod~$p$ Galois representations.
  We assume that the Galois representation is irreducible and odd (in the symplectic sense). The modularity condition is formulated
  using the \'etale and the algebraic de Rham cohomology of Siegel modular varieties of level prime to~$p$. We concentrate
  on the case when the Galois representation is ordinary at~$p$ and we give a corresponding list of Serre weights. When
  the representation is moreover tamely ramified at~$p$, we conjecture that all weights of this list are modular,
  otherwise we describe a subset of weights on the list that should be modular. We propose a construction of de Rham
  cohomology classes using the dual BGG complex, which should realise some of these weights.
\end{abstract}

\selectlanguage{francais}

\section{Introduction}\label{sec:intro}

Dans cet article, on pr\'esente une conjecture de type de Serre sur la modularit\'e d'une repr\'esentation galoisienne modulo~$p$ de rang $4$ \`a valeurs dans le groupe symplectique 
$\rhob:\Gamma={\rm Gal}(\overline{\Q}/\Q)\rightarrow \GSp_4(\fpb)$. Nous supposons la repr\'esentation irr\'eductible et motiviquement impaire (c'est-\`a-dire, que $p=2$ ou bien que les espaces propres pour l'action de la conjugaison complexe pour les deux valeurs propres $1$ et $-1$ sont des plans lagrangiens).
Nous proposons une conjecture sur les poids, mais nous n'abordons pas la question des niveaux de la vari\'et\'e de Siegel dans la cohomologie de laquelle $\rhob$ appara\^it.

La question de la modularit\'e d'une telle repr\'esentation, au moins lorsque $\rhob\vert_{I_p}$ est suppos\'ee ordinaire et de Fontaine--Lafaille de poids
$p$-petits, a \'et\'e pos\'ee par l'un des auteurs en 1995 (\cite{T0} et sect.~\ref{sec:serre_conj} ci-dessous). Dans sa th\`ese, le
premier auteur a formul\'e pour ${\GL_n}_{/\Q}$ une conjecture de modularit\'e de type de Serre, en donnant la liste compl\`ete des poids
r\'eguliers possibles lorsque $\rhob|_{I_p}$ est semi-simple. Le formalisme qu'il a d\'evelopp\'e est valide pour des groupes r\'eductifs
plus g\'en\'eraux. Il produit des syst\`emes locaux en $\fpb$-vectoriels associ\'es \`a des $\fpb$-repr\'esentations irr\'eductibles de $\GL_n(\fp)$
sur les espaces localement sym\'etriques de niveau premier \`a~$p$ dans la cohomologie desquels devraient exister des classes propres pour
le syst\`eme de valeurs propres de Hecke associ\'e \`a $\rhob$. Ces syst\`emes locaux, appel\'es poids de Serre, remplacent le poids
$k(\rhob)$ de la conjecture de Serre \cite{Se}.

Nous discutons maintenant quelques conjectures et r\'esultats dans le cas de ${\GSp_4}_{/\Q}$ qui expliquent plusieurs poids de cette conjecture de type de
Serre. On formule la condition de modularit\'e en utilisant la cohomologie
\'etale et la cohomologie de de Rham alg\'ebrique des vari\'et\'es de Siegel de niveau premier \`a~$p$, et on se concentre sur le cas o\`u
$\rhob$ est ordinaire en~$p$.
 Dans le cas o\`u $\rhob\vert_{I_p}$ est de plus mod\'er\'ement ramifi\'ee, nous explicitons les poids de Serre pr\'edits. En
g\'en\'eral, notre liste compte vingt \'el\'ements. Parmi ceux-ci, huit ont une interpr\'etation en termes de formes compagnons
($p$-ordinaires). Ils correspondent de mani\`ere naturelle \`a des twists de $\rhob$ par des puissances convenables du caract\`ere
cyclotomique modulo~$p$ comme pour les formes compagnons dans la conjecture de Serre \cite{G}.
Quant aux douze autres poids on peut au moins exhiber, dans le cas o\`u la restriction $\rhob|_{D_p}$ \`a tout le groupe de d\'ecomposition est totalement d\'ecompos\'ee, des rel\`evements cristallins avec les poids de Hodge--Tate correspondants (voir la proposition~\ref{prop:cryst_lifts} ci-dessous).

Par ailleurs, lorsque $\rhob$ est $p$-ordinaire non n\'ecessairement mod\'er\'ee, d'exposants $i_0=0\leq i_1<i_2\leq
i_3<p-1$, nous conjecturons l'existence d'une forme cuspidale holomorphe propre $p$-ordinaire de poids $(k,\ell)$ pour
$\ell=i_1+2$ et $k=i_2+1$, de niveau premier \`a~$p$ et de repr\'esentation galoisienne r\'esiduelle~$\rhob$.  Sous des
hypoth\`eses de d\'ecomposabilit\'e partielle de $\rhob\vert_{I_p}$, il y a trois twists de $\rhob$ pour lesquelles nous
conjecturons qu'ils sont associ\'es \`a des formes cuspidales cohomologiques (holomorphes ou de Whittaker) de niveau premier
\`a~$p$ et ordinaires en $p$. Il y a quatre autres twists de $\rhob$ pour lesquels nous conjecturons aussi qu'ils sont
associ\'es \`a des formes cuspidales cohomologiques ordinaires, mais qui dans notre approche g\'eom\'etrique semblent moins
accessibles.  

Comme on le d\'ecrira avec plus de d\'etails dans la section~\ref{sec:twists}, ces quatre premiers twists correspondent de
mani\`ere naturelle aux repr\'esentants de Kostant dans le groupe de Weyl de $\GSp_4$ modulo le groupe de Weyl du sous-groupe
de Levi du parabolique de Siegel. Dans un travail r\'ecent~\cite{T2}, le deuxi\`eme auteur d\'emontre l'existence d'une forme
cuspidale $p$-adique (et conjecturalement classique) dans le Cas~1, sous certaines hypoth\`eses globales. La m\'ethode
consiste \`a g\'en\'eraliser le travail de Faltings--Jordan en utilisant le complexe BGG dual (introduit par~\cite{FC}, et
\'etudi\'e dans \cite{MT} et \cite{PT}). Il se peut que les deux autres cas non-triviaux puissent se d\'emontrer avec la
m\^eme m\'ethode, mais avec des difficult\'es techniques suppl\'ementaires.

Cependant, notons que, apr\`es une version ant\'erieure de cet article, notre conjecture pour les huit poids qu'on vient de
mentionner a \'et\'e r\'ecemment \'etablie par Gee et Geraghty \cite{GG} par une m\'ethode de rel\`evement automorphe d'une
repr\'esentation galoisienne r\'esiduelle modulaire.

Notons que \cite{Y2} \'etudie une question pr\'ealable \`a des g\'en\'eralisations ult\'erieures de la conjecture de type de
Serre: \'etant donn\'e $H$ r\'eductif sur $\Q_p$ et ${\rm Gal}(\overline{\Q}/\Q)\rightarrow H(\overline{\Q}_p)$ un
homomorphisme \og g\'eom\'etrique\fg, quel est le meilleur groupe r\'eductif $G_{/\Q}$ sur lequel chercher une repr\'esentation
automorphe $\pi$ et une repr\'esentation $r:{}^L G\rightarrow H$ telles que $\rho=r\circ\rho_\pi$ ($\rho_\pi$ associ\'ee \`a
$\pi$ par la correspondance de Langlands). Par exemple, pour nous si $H=\GSp_4$, la r\'eponse est $G=\GSp_4$ et $r=\spin$.

Le plan de l'article est le suivant. Dans la section~\ref{sec:notations} on introduit les notations li\'ees au groupe $\GSp_4$ et on discute la condition d'ordinarit\'e dans la section~\ref{sec:ord}. Puis on \'enonce les conjectures de type de Serre dans la section~\ref{sec:serre_conj}, on donne une explicitation de la liste des poids et on d\'ecrit les rel\`evements cristallins correspondants dans le cas totalement d\'ecompos\'e. Ensuite on \'etudie dans la section~\ref{sec:twists} en d\'etail les twists de $\rhob$ qui correspondent \`a des formes compagnons et \'enonce les conjectures correspondants. Finalement, en 
section~\ref{sec:bgg} on esquisse la m\'ethode, utilisant le complexe BGG 
dual, qui fournit une forme compagnon (modulo $p$ puis un rel\`evement $p$-adique) dans le premier
cas.

\textbf{Remerciements:} Le premier auteur tient \`a remercier Matthew Emerton, Toby Gee et David Geraghty pour des discussions utiles, ainsi
que l'I.H.\'E.S., l'Universit\'e de Paris~7 et l'IAS o\`u une partie de ce travail a \'et\'e r\'ealis\'ee. Le second auteur remercie
Kyoto University, Columbia University et l'Academia Sinica de Taipei o\`u une partie de ce travail a \'et\'e r\'ealis\'ee,
pour leur hospitalit\'e. Les auteurs remercient enfin le rapporteur dont les remarques ont am\'elior\'e la qualit\'e de cet article.

\section{Notations}\label{sec:notations}

Soit $(V,\psi)$ un $\Z$-module symplectique unimodulaire de rang $4$, de base $(e_1,e_2,e_3,e_4)$ avec $\psi(e_1,e_4)=\psi(e_2,e_3)=1$ et $\psi(e_i,e_j)=0$
si $i\leq j$ autres que $(1,4)$ et $(2,3)$. La matrice de $\psi$ dans cette base est
$$J=\left(\begin{array}{cccc}0&0&0&1\\0&0&1&0\\0&-1&0&0\\-1&0&0&0\end{array}\right)$$
Soit $G=\GSp_4(V,\psi)=\{X\in \GL_4 : {}^t XJX=\nu\cdot J\}$ le sch{\'e}ma en groupes r\'eductif connexe sur $\Z$ des
similitudes de $(V,\psi)$. Le facteur de similitude $\nu$ d{\'e}finit un caract{\`e}re $\nu_\simi:G\rightarrow \G_m$; son noyau est
le $\Z$-sch{\'e}ma en groupes $\Sp_4$.

\label{eq:s} Soit $s=\left(\begin{array}{cc}0&1\\1&0\end{array}\right)$, $s'=\left(\begin{array}{cc}0&1\\-1&0\end{array}\right)$, 
et $1_2=\left(\begin{array}{cc}1&0\\0&1\end{array}\right)$.
On note $B=TN$, resp.\ $Q=MU$, $P=M_1U_1$ les d{\'e}compositions de Levi standard du sous-groupe de Borel $B$ de $G$ stabilisateur du drapeau $\langle e_1\rangle\subset \langle e_1,e_2\rangle$, du parabolique de Siegel $Q$ stabilisateur du plan isotrope $ \langle e_1,e_2\rangle$,
resp.\ du parabolique de Klingen~$P$ stabilisateur de la droite $ \langle e_1\rangle$. On a
$$T=\left\{\left(\begin{array}{cccc}t_1&0&0&0\\0&t_2&0&0\\0&0&\nu\cdot t_2^{-1}&0\\0&0&0&\nu\cdot t_1^{-1}\end{array}\right)\right\},$$
$$N=\left\{\left(\begin{array}{cccc}1&x&*&*\\0&1&*&*\\0&0&1&-x\\0&0&0&1\end{array}\right)\right\}\cap G,$$
$$Q=\left\{\left(\begin{array}{cc}A&C\\0&D\end{array}\right)\right\}\cap G,$$
$$M=\left\{\left(\begin{array}{cc}A&0\\0&D\end{array}\right); {}^t \!AsDs=\nu\cdot 1_2\right\},$$
$$U=\left\{\left(\begin{array}{cc}1_2&C\\0&1_2\end{array}\right); sC\,\text{sym{\'e}trique}\right\},$$
$$P=\left\{\left(\begin{array}{ccc}a&*&*\\0&A&*\\0&0&{\rm det}\,A\cdot a^{-1}\end{array}\right); A\in \GL_2\right\}\cap G.$$
Notons que dans cette notation certaines {\'e}toiles (resp.\ certains z{\'e}ros) repr{\'e}sentent des matrices $1\times 2$ ou $2\times 1$.
De m{\^e}me, avec ces notations, on a
$$M_1=\left\{\left(\begin{array}{ccc}a&0&0\\0&A&0\\0&0&{\rm det}\,A\cdot a^{-1}\end{array}\right); A\in \GL_2\right\},$$
$$U_1=\left\{\left(\begin{array}{cccc}1&x&y&*\\0&1&0&y\\0&0&1&-x\\0&0&0&1\end{array}\right)\right\}.$$

 Soit $W_G=N_G(T)/T$ le groupe de Weyl de $(G,T)$. Son action sur le groupe des caract{\`e}res est donn{\'e}e par
$w\cdot \lambda (t)=\lambda(w^{-1}tw)$.
Il est engendr{\'e} par (les classes de) $s_0=\left(\begin{array}{cccc}s&0_2\\0_2&s\end{array}\right)$
et
$s_1=\left(\begin{array}{ccc}1&0&0\\0&s'&0\\0&0&1\end{array}\right)$.  Il admet la pr{\'e}sentation:
$$W_G=\langle s_0,s_1;s_0^2, s_1^2,(s_0 s_1)^4 \rangle.$$

Le groupe de Weyl $W_M$ de $M$ est engendr{\'e} par $s_0$. On a un syst{\`e}me de repr{\'e}sentants de Kostant
de $W_M\backslash W_G$ donn{\'e} par $W^M=\{1_4,s_1,s_1s_0,s_1 s_0 s_1\}$.

Soit $X(T)$ (resp.\ $Y(T)$) le groupe des caract{\`e}res (resp.\ cocaract{\`e}res) de $T$. On identifie $X(T)$ au r{\'e}seau de $\Z^3$ des triplets
$(a,b;c)\in\Z^3$ avec $c\equiv a+b\pmod 2$ par la formule
\[ \lambda:t=\diag(t_1,t_2,\nu t_2^{-1},\nu t_1^{-1})\mapsto t_1^a t_2^b \nu^{(c-a-b)/ 2}. \] 
Notons qu'on a alors $\lambda(\diag(z,z,z,z))=z^c$.  Avec ces notations, notre choix des racines simples de $G$ est $\alpha_0=(1,-1;0)$
et $\alpha_1=(0,2;0)$; $\alpha_0$ est la racine courte. De plus, le facteur de similitude $\nu_\simi : G \to \G_m$ d\'efini le poids de coordonn\'ees $(0,0;2)$.
Notons $X(T)_+ = \{(a,b;c) \in X(T) : a \ge b \ge 0 \}$ l'ensemble des poids dominants.
Les $\alpha_i$ correspondent aux r{\'e}flexions $s_i$. L'{\'e}l{\'e}ment $s_1 s_0$,
resp.\  $s_1 s_0 s_1$, envoie $(a,b;c)$ sur $(b,-a;c)$, resp.\ sur $(-b,-a;c)$.

\medskip

Soit $(\widehat{G},\widehat B,\widehat T)$ le groupe r{\'e}ductif dual de $(G,B,T)$, d\'etermin\'e \`a isomorphisme pr\`es. Fixons
un isomorphisme entre les donn\'ees radicielles bas\'ees $\psi_0(\widehat G) \cong \psi_0(G)^*$. Alors le groupe $W_G$ s'identifie canoniquement 
au groupe de Weyl $W_{\widehat{G}}=N_{\widehat{G}}(\widehat{T})/\widehat{T}$ de $\widehat{G}$.

\label{spin}Fixons de plus des \'epinglages de $(G,B,T)$ et de $(\widehat{G},\widehat B,\widehat T)$. Alors l'isomorphisme ${\spin}: \widehat{G}\cong G$
respecte ces \'epinglages et induit un certain isomorphisme $\psi_0(\widehat G) \cong \psi_0(G)$ sur les donn\'ees radicielles bas\'ees, ce qui
le caracterise. On donne une formule explicite pour ce dernier isomorphisme dans la section~\ref{sec:serre_conj},~\'equation \eqref{eq:11}.
L'isomorphisme $\spin$ applique le groupe de Weyl $W_{\widehat{G}}$ isomorphiquement sur $W_G=N_G(T)/T$.
Le compos{\'e}
\[ \iota : W_G = W_{\widehat G} \stackrel{\spin}{\longrightarrow} W_G \]
induit un automorphisme qui \'echange les deux g{\'e}n{\'e}rateurs $s_0$ et $s_1$. Notons que ceci est compatible avec l'invariance de la
pr{\'e}sentation de $W_G$ par {\'e}change de $s_0$ et $s_1$. Voir \cite{MT} pour les d{\'e}tails. Ainsi, via $\iota$, $s_1 s_0$ agit par
$(a,b;c)\mapsto(-b,a;c)$, et $s_1 s_0 s_1$ agit par $(a,b;c)\mapsto(-a,b;c)$.

\section{Ordinarit{\'e}}\label{sec:ord}

Consid{\'e}rons une repr{\'e}sentation cuspidale $\pi$ de $\GSp_4(\A)$ dont la composante {\`a} l'infini est dans la s{\'e}rie discr{\`e}te holomorphe de param{\`e}tre de Harish-Chandra $(a+2,b+1;a+b)$ avec $a\geq b\geq 0$; en posant $k=a+3$ et $\ell=b+3$, la construction de Harish-Chandra permet de lui associer des formes modulaires holomorphes de poids $(k,\ell)$ avec $k\geq \ell\geq 3$. On fixe une telle forme, que l'on note $f$.

Soit $\Gamma={\rm Gal}(\overline{\Q}/\Q)$; fixons un nombre premier~$p$ et consid{\'e}rons 
la repr{\'e}sentation galoisienne
$$\rho_{\pi,p}=\rho_{f,p}:\Gamma\rightarrow \GSp_4(\overline{\Q}_p)$$
non ramifi{\'e}e hors de ${\rm Ram}(\pi)\cup \{p,\infty\}$ et telle que le polyn{\^o}me caract{\'e}ristique du Frobenius
arithm{\'e}tique $\phi_\ell$ en $\ell\neq p$, $\ell\not\in {\rm Ram}(\pi)$ soit le polyn{\^o}me de Hecke $P_{\pi,\ell}(X)$,
c'est-{\`a}-dire le polyn{\^o}me tel que $P_{\pi,\ell}(\ell^{-s})$ soit l'inverse du facteur d'Euler en $\ell$ de la fonction $L$
spinorielle de $\pi$ d{\'e}cal\'ee de ${k+\ell-3\over 2}$. On note aussi $P_{f,\ell}(X)=P_{\pi,\ell}(X)$.

L'existence de cette repr{\'e}sentation r{\'e}sulte des travaux de R.~Taylor \cite{Ta}, Laumon \cite{La} et Weissauer \cite{We}, 
\cite{We2}, dont la derni\`ere r\'ef\'erence \'etablit la symplecticit\'e.
%mis {\`a} part la symplecticit{\'e}, qui r{\'e}sulterait de travaux non publi{\'e}s de Weissauer ou de travaux {\`a} para\^ itre de J.~Arthur. On la suppose acquise.

Soit $\cH_q$ la $\Z$-alg{\`e}bre de Hecke sph{\'e}rique de $G$ en $q$ (c'est le $\Z$-module des fonctions {\`a} support compact sur $G(\Q_q)$, biinvariantes
par $G(\Z_q)$). On note $T_{q,1}$, $T_{q,2}$, resp.\ $T_{q,0}$ les g{\'e}n{\'e}rateurs de $\cH_q$ associ{\'e}s {\`a} $\diag(1,1,q,q)$, $\diag(1,q,q,q^2)$, 
resp.\ $\diag(q,q,q,q)$.

Pour tout premier $q$ o{\`u} $\pi$ est non ramifi{\'e}e, on d{\'e}finit les valeurs propres $a_{q,i}$ des op{\'e}rateurs de Hecke $T_{q,i}$ ($i=1,2$). 

\begin{de} \label{df:ord} Supposons que $\pi$ est non ramifi{\'e} en~$p$.
  On dit que $\pi$ est \emph{ordinaire en~$p$} (ou \emph{$p$-ordinaire}) si l'une des conditions \'equivalentes
  suivantes est satisfaite. 

\begin{itemize}
\item $\ord_p(a_{p,1}) = 0$ et $\ord_p(a_{p,2}) = b =\ell-3$. 

\item On peut ordonner les racines du polyn\^ome de Hecke
\[ P_{\pi,p}(X) =P_{f,p}(X)= (X-\alpha)(X-\beta)(X-\gamma)(X-\delta) \]
de sorte que $\ord_p(\alpha) = 0$, $\ord_p(\beta) = \ell-2$, $\ord_p(\gamma) = k-1$ et $\ord_p(\delta) = k+\ell-3$.
\end{itemize}
L'\'equivalence provient de la formule $P_{f,p}(X)=X^4-a_{p,1}X^3+(pa_{p,2}+2p^{k+\ell-3}\eta(p))X^2-p^{k+\ell-3}a_{p,1}\eta(p)X+(p^{k+\ell-3}\eta(p))^2$, o\`u $\eta$ d\'esigne le caract\`ere de Dirichlet donnant la partie finie du caract\` ere central de $\pi$ 
(on l'appelle aussi le Nebentypus de $f$); pour cette formule, voir sect.~3.5 de \cite{Ta0}. 
On dit aussi que la forme $f$ est $p$-ordinaire. 
\end{de}

\begin{rk}
  Les valuations des racines de $P_{f,p}(X)$ sont deux \`a deux distinctes par la condition $k \ge \ell \ge 3$. De plus,
  $a_{p,1} \equiv \alpha \pmod {\mathfrak{m}_{\zpb}}$.
\end{rk}

Soit $\cO$ un anneau de valuation discr{\`e}te fini et plat sur $\Z_p$ contenu dans $\overline{\Q}_p$ suffisamment grand pour
que $\rho$ soit d{\'e}finie dessus, c'est-{\`a}-dire qu'il existe un $\cO$-r{\'e}seau $L$ stable par $\rho$. Soit $\varpi$ un
param{\`e}tre uniformisant de $\cO$, $\kappa=\cO/(\varpi)$ et soit $\overline{\rho}_{f,p}:\Gamma\rightarrow \GSp_4(\kappa)$ la
repr{\'e}sentation de $\Gamma$ sur $L/\varpi L$.
 
Soit $D_p$ un groupe de d{\'e}composition en~$p$ dans $\Gamma$; soit $I_p\subset D_p$ son sous-groupe d'inertie.  Soit
$\epsilon$, resp.\ $\omega$ le caract{\`e}re cyclotomique $p$-adique, resp.\ modulo~$p$. Pour tout nombre $u$ ($p$-adique ou mod~$p$) on
note $\nr(u)$ le caract\`ere non-ramifi\'e de~$D_p$ qui envoie le
Frobenius arithm\'etique sur~$u$.
Si $\pi$ est ordinaire en~$p$
(\cite{Hi}, \cite{T1}), on a par \cite{U}:
$$\rho_{f,p}\vert _{D_p}\sim \left( \begin{array}{cccc}\epsilon^{k+\ell-3}{\rm nr}\Big({\delta\over p^{k+\ell-3}}\Big)&*&*&*\\0&\epsilon^{k-1}
{\rm nr}\Big({\gamma\over p^{k-1}}\Big)&*&*\\0&0&\epsilon^{\ell-2}{\rm nr}\Big({\beta\over p^{\ell-2}}\Big)&*\\0&0&0&
{\rm nr}(\alpha)\end{array}\right),$$
donc
$$\rho_{f,p}\vert_{I_p}\sim \left(\begin{array}{cccc}\epsilon^{k+\ell-3}&*&*&*\\0&\epsilon^{k-1}&*&*\\0&0&\epsilon^{\ell-2}&*\\0&0&0&1\end{array}\right)$$
et aussi
$$\overline{\rho}_{f,p}\vert _{I_p}\sim \left(\begin{array}{cccc}\omega^{k+\ell-3}&*&*&*\\0&\omega^{k-1}&*&*\\0&0&\omega^{\ell-2}&*\\0&0&0&1\end{array}\right).$$

\section{Conjecture de type de Serre pour $\GSp_4$}\label{sec:serre_conj}

Soit $S$ un ensemble fini de places de $\Q$ contenant $p$ et $\infty$. Pour chaque $\ell\in S$, on fixe un sous-groupe d'inertie $I_\ell\subset \Gamma$.
Soit $\kappa$ un corps fini de caract{\'e}ristique~$p$.
Fixons une fois pour toutes un isomorphisme $\qpb \congto \C$.

Soit $\overline{\rho}:\Gamma\rightarrow \GSp_4(\kappa)$ une repr{\'e}sentation $S$-ramifi{\'e}e, c'est-{\`a}-dire 
un homomorphisme continu tel que $\overline{\rho}(I_\ell)=1$ pour tout nombre premier $\ell\not\in S$. 
Supposons que $\overline{\rho}$ soit absolument irr{\'e}ductible.
Dans \cite{Ta0}, l'un des auteurs a conjectur{\'e} que, sous l'hypoth{\`e}se que $\overline{\rho}$ est modulaire (c'est-{\`a}-dire provient d'une forme cuspidale cohomologique $\pi$ non ramifi\'ee en $p$ 
pour $\GSp_4(\A)$) et que sa restriction {\`a} $D_p$ est 
triangulaire sup{\'e}rieure avec des caract{\`e}res sur la diagonale deux {\`a} deux distincts,
il en est de m{\^e}me pour ses d{\'e}formations de m{\^e}me type (un ordre des caract{\`e}res de la diagonale {\'e}tant fix{\'e}).

Il a introduit (\cite{T0}, section~9) une condition n{\'e}cessaire de modularit{\'e} pour $\overline{\rho}$, appel{\'e}e l'imparit{\'e} motivique:
 $\overline{\rho}$ est dite \emph{motiviquement impaire} si $p=2$ ou
$\nu_\simi\circ\overline{\rho}(c)=-1$ (pour une conjugaison complexe $c$) ou, de mani{\`e}re {\'e}quivalente si
$$\overline{\rho}(c)\sim\left(\begin{array}{cccc}1&0&0&0\\0&1&0&0\\0&0&-1&0\\0&0&0&-1\end{array}\right),
$$
la conjugaison ayant lieu dans $\GSp_4$.

Il a pos{\'e} la question suivante (Cours au Tata Institute, 1995, non publi{\'e}): si la repr{\'e}sentation
$\overline{\rho}$ est absolument irr{\'e}ductible, motiviquement impaire et $p$-ordinaire d'exposants $i_0\leq i_1\leq
i_2\leq i_3$ (voir l'\'equation \eqref{eq:16} ci-dessous) o\`u les $i_j$ sont deux {\`a} deux distincts avec $i_3-i_0<p-1$,
provient-elle d'une repr{\'e}sentation cuspidale de $\GSp_4(\A)$ avec $\pi_\infty$ dans la s{\'e}rie discr{\`e}te?

On peut formuler une conjecture l{\'e}g{\`e}rement plus g{\'e}n{\'e}rale et plus pr{\'e}cise. Supposons que la
repr{\'e}sentation $\overline{\rho}$ est $p$-ordinaire. On a donc {\`a} conjugaison pr{\`e}s dans $\GL_4$:
\begin{equation}
  \label{eq:16}
  \overline{\rho}\vert _{D_p}\sim \left( \begin{array}{cccc}\omega^{i_3}{\rm nr}(u_3)&*&*&*\\0&\omega^{i_2}{\rm nr}(u_2)&*&*\\0&0&\omega^{i_1}{\rm nr}(u_1)&*\\
      0&0&0&\omega^{i_0}{\rm nr}(u_0)\end{array}\right).
\end{equation}

Apr{\`e}s torsion par une puissance de $\omega$, on peut supposer que les exposants satisfont $i_0=0\leq i_1\leq i_2\leq i_3$ et $i_j\leq j(p-2)$ pour $j=1$, $2$, $3$.
On a n\'ecessairement $i_0 + i_3 = i_1 + i_2$. Cette relation est \'evidente sous l'hypoth\`ese $i_3 < p-1$, mais en fait, elle est toujours satisfaite car on peut toujours supposer que la conjugaison~\eqref{eq:16} a lieu dans $\GSp_4$, comme nous l'a fait remarquer D.\ Prasad (lettre \`a l'un des auteurs).

Notons que ce choix d'entiers croissants $i_j\in\Z$ n'est pas unique en g\'en\'eral, m\^eme si on impose $i_j\leq j(p-2)$.
Une fois un tel ordre fix\'e, supposons en outre que $i_1<i_2$; il existe alors un couple unique d'entiers $(k,\ell)$ avec
$k\geq \ell\geq 2$ tel que $i_1=\ell-2$ et $i_2=k-1$.  Si de plus les $i_j$ sont deux \`a deux
distincts, on a m{\^e}me $\ell\geq 3$ de sorte qu'il existe $(a,b)$ avec $a\geq b\geq 0$ tel que $i_1=b+1$ et $i_2=a+2$, et
$k=a+3$, $\ell=b+3$.

Nous appelerons $(k,\ell)$ le poids modulaire du couple $(\overline{\rho},(i_j)_{j=1,\ldots,3})$ constitu{\'e} de la repr{\'e}sentation $p$-ordinaire $\overline{\rho}$ et de la suite ordonn{\'e}e $(0,i_1,i_2,i_3)$ des exposants.

\begin{de} On dira que les exposants sont \emph{$p$-petits} si apr{\`e}s torsion par une puissance de $\omega$, on a
  $0=i_0\leq i_1\leq i_2 \leq i_3<p-1$ .
\end{de}

Dans ce cas, le couple d'entiers $(k,\ell)$ satisfait la condition $k+\ell-3<p-1$.

\begin{conj}\label{conj:classical}
Soit $(\overline{\rho},(i_j)_j)$ un couple constitu\'e d'une repr\'esentation irr{\'e}ductible motiviquement impaire et $p$-ordinaire 
et d'un ordre des exposants avec $i_0=0$, $i_1<i_2$. Supposons ces exposants $p$-petits. Soit $(k,\ell)$ son poids modulaire. Supposons
 que $\rhob$ soit de Fontaine--Laffaille en~$p$ \(de poids dans $[0,p-2]$\); il existe alors une forme cuspidale holomorphe $f$ de poids $(k,\ell)$, $p$-ordinaire et de niveau premier {\`a}~$p$, associ{\'e}e {\`a} $\overline{\rho}$, au sens que $\overline{\rho}_{f,p}$ est isomorphe {\`a} $\overline{\rho}$.
\end{conj}

\begin{rk}
  On peut montrer que lorsque les exposants sont deux \`a deux distincts, la repr\'esentation $\rhob$ est de
  Fontaine--Laffaille en~$p$ de poids dans $[0,p-2]$, en bref $FL^{[0,p-2]}$, si et seulement si les \'etoiles de la
  premi\`ere surdiagonale de $\rhob\vert_{D_p}$ sont peu ramifi\'ees au sens de Serre. Cette condition se produit toujours,
  sauf s'il existe $j$ tel que $i_{j+1}=i_j+1$ et $u_{j+1}=u_j$ (lorsqu'aucun tel $j$ n'existe, on peut trouver la construction d'un rel\`evement cristallin
dans le lemme 7.6.7 de \cite{GG}, par exemple). 
Ainsi la condition $FL^{[0,p-2]}$ est \og g\'en\'eriquement
  satisfaite\fg.  
\end{rk}

\vskip 3mm

\noindent{\bf Pr{\'e}ambules {\`a} la th{\'e}orie des poids de Serre:}

Lorsque $\rhob$ est de Fontaine--Laffaille en~$p$, la forme $f$ conjecturale intervient dans la cohomologie de de Rham de la
vari{\'e}t{\'e} de Siegel $X$ de niveau premier {\`a}~$p$, {\`a} coefficients dans le fibr{\'e} {\`a} connexion associ{\'e}
{\`a} la repr{\'e}sentation $V_\lambda$ de $\GSp_4$ de plus haut poids $\lambda=(a,b;a+b)$ dans les notations de \cite{FC},
chap.~VI. Cette remarque est en accord avec la terminologie de \cite{T1}, par exemple, o\`u l'on dit qu'un poids $(k,\ell)$
avec $k \ge \ell \ge 3$ est cohomologique.

Par des th{\'e}or{\`e}mes de comparaison classiques, le syst{\`e}me de valeurs propres de Hecke de $f$ intervient aussi dans la cohomologie {\'e}tale $H^3_{et}(X\otimes \overline{\Q},V_\lambda(\overline{\Q}_p))$ du syst{\`e}me local associ{\'e} {\`a} la repr{\'e}sentation $V_\lambda$. Lorsque la 
repr\'esentation r\'esiduelle $\overline{\rho}$ est absolument irr\'eductible, pour tout choix fix\'e d'un $\zp$-r\'eseau stable $V_\lambda$,
il en r{\'e}sulte que la r{\'e}duction modulo~$p$ du r{\'e}seau 
$${\rm Im}\left(H^3_{et}(X\otimes \overline{\Q},V_\lambda(\overline{\Z}_p))\rightarrow H^3_{et}(X\otimes \overline{\Q},V_\lambda(\overline{\Q}_p))\right)$$
contient la contragr{\'e}diente $\overline{\rho}^\vee$ de $\overline{\rho}$. 

On se propose de donner dans ce qui va suivre une formulation de la question de la modularit{\'e} cohomologique de $\overline{\rho}$ purement en caract{\'e}ristique $p$, sans aborder la question de l'existence d'un rel{\`e}vement automorphe classique en caract{\'e}ristique z{\'e}ro. Il faut cependant {\it a priori} distinguer la question de l'existence d'une classe de cohomologie de de Rham modulo~$p$ ou d'une classe de cohomologie {\'e}tale {\`a} coefficients dans des syst{\`e}mes locaux en $\overline{\F}_p$-vectoriels, qui repr{\'e}sente $\overline{\rho}$. On dit qu'une classe de cohomologie de de Rham $c$ repr{\'e}sente $\overline{\rho}$ si elle est propre pour les correspondances de Hecke pour tout $\ell\notin S\cup\{p\}$ premier, et que le polyn{\^o}me de Hecke $P_{c,\ell}(X)$ est le polyn{\^o}me caract{\'e}ristique de l'image du Frobenius (g{\'e}om{\'e}trique) par la contragr{\'e}diente $\overline{\rho}^\vee$.

La comparaison entre les deux formulations, celle en termes d'une classe de cohomologie de de Rham, et celle en termes d'une classe de cohomologie {\'e}tale, repose sur la validit{\'e} du th{\'e}or{\`e}me de comparaison {\'e}tale--cristallin modulo~$p$ \cite{Fa} de Faltings. Les poids permis {\it a priori} pour ce th{\'e}or{\`e}me de comparaison modulo~$p$, sont les $(k,\ell)$ tels que $k\geq \ell\geq 3$ avec $k+\ell-3<p-1$ qui correspondent, comme on le verra, \`a des points entiers de l'alc{\^o}ve fondamentale, alors que la conjecture ci-dessous qui donne les poids de Serre possibles pour une repr{\'e}sentation r{\'e}siduelle $\overline{\rho}$ fait intervenir des poids $p$-restreints qui peuvent {\^e}tre n{\'e}anmoins hors de cette alc{\^o}ve. Pour ces poids, la comparaison modulo~$p$ n'est pas valable, {\`a} moins que ces classes ne se rel{\`e}vent en caract{\'e}ristique nulle car la limitation sur les poids pour le th{\'e}or{\`e}me de comparaison dispara{\^\i}t alors.

Le travail r{\'e}cent d'un des auteurs \cite{He}, a permis de formuler une conjecture presque compl{\`e}te dans le cas de
$\GL_n$ lorsque la repr{\'e}sentation galoisienne r{\'e}siduelle est mod{\'e}r{\'e}e en~$p$.  Elle est exprim{\'e}e en termes
de classes de cohomologie singuli{\`e}re {\`a} coefficients dans des syst{\`e}mes locaux en $\overline{\F}_p$-vectoriels, et
non pas en termes de cohomologie {\'e}tale, puisqu'il n'y a pas en g{\'e}n{\'e}ral d'action de Galois sur la cohomologie des
espaces localement sym{\'e}triques, non alg{\'e}briques, associ{\'e}s {\`a} $\GL_n$.

Dans le cas de $G=\GSp_4$, l'objet de la pr{\'e}sente section est de sp{\'e}cifier, sous des hypoth{\`e}ses similaires, les
syst{\`e}mes locaux {\'e}tales en $\overline{\F}_p$-vectoriels sur la vari{\'e}t{\'e} de Siegel sur $\Q$ de niveau premier
{\`a}~$p$ dans la cohomologie desquels la contragr{\'e}diente $\overline{\rho}^\vee$ intervient. Ces syst{\`e}mes locaux sont
associ{\'e}s {\`a} des repr{\'e}sentations irr{\'e}ductibles du groupe fini $G(\F_p)$ sur $\overline{\bF}_p$ qu'on appelera
poids de Serre. {\it A posteriori}, on peut interpr{\'e}ter la question de 1995 comme celle de l'existence de l'une d'entre
elles, de plus haut poids $(a,b;a+b)$. Nous donnons ici la liste compl{\`e}te des poids de Serre r\'eguliers (voir d\'ef.~\ref{df:reg}) pour
$\overline{\rho}$ $p$-ordinaire mod{\'e}r{\'e}e.

Nous commen\c cons en posant la d{\'e}finition suivante.

\begin{de}
  Un \emph{poids de Serre} est une repr{\'e}sentation irr{\'e}ductible du groupe fini $G(\fp)$ sur $\fpb$, {\`a} isomorphisme pr{\`e}s.
\end{de}

Cette notion apparaissait implicitement dans~\cite{AS} pour $\GL_2$ puis a \'et\'e utilis\'ee syst{\'e}matiquement dans~\cite{ADP} et~\cite{BDJ}.

\emph{Jusqu'\`a la prop.~\ref{prop:conj_combi} on consid{\'e}rera $G$ comme groupe sur $\F_p$.}
Pour $\lambda \in X(T)$ dominant, le $G$-module dit de Weyl dual est d{\'e}fini par l'induction alg\'ebrique
\begin{align*}
W(\lambda) &= \Ind_{B^-}^G (\fpb(\lambda)) \\
 &= \{ f \in \Mor(G,\A^1) : f(bg) = \lambda(b) f(g)\, \forall \,g\in G,\ b\in B^-\}
\end{align*}
o{\`u} $B^-$ d{\'e}signe le Borel oppos{\'e} de $B$. Le plus grand $G$-sous-module semisimple $F(\lambda) := \soc_G W(\lambda)$ est isomorphe au $G$-module simple de plus haut poids $\lambda$ \cite{J1}, II.2.4. 

\begin{de}
   On introduit le sous-ensemble $X_1(T)$ de $X(T)$ des \emph{poids $p$-restreints},
   \begin{align*}
     X_1(T)&=\{\lambda \in X(T): 0 \leq \langle \lambda, \alpha_i^\vee\rangle < p \quad \forall i \} \\
     &= \{ (a,b;c) \in X(T) : 0 \le a-b < p,\ 0 \le b < p \},
   \end{align*}
   ainsi que le sous-groupe 
   \begin{align*}
     X^0(T) &=\{\lambda \in X(T): \langle \lambda, \alpha_i^\vee\rangle = 0 \quad \forall i \} \\
     &= \{ (0,0;c) : c \in 2\Z \}.
   \end{align*}
\end{de}

Comme $G' \cong \Sp_4$ est simplement connexe, on a par une g{\'e}n{\'e}ralisation simple d'un th{\'e}or{\`e}me classique de
Steinberg (voir~\cite{He}, prop.~1.3 dans l'appendice).

\begin{pro}
  Tout poids de Serre est la restriction aux points $\fp$-rationnels d'un $G$-module simple $F(\lambda)$ avec $\lambda \in X_1(T)$.
  Si $\lambda$, $\lambda' \in X_1(T)$ on a $F(\lambda) \cong F(\lambda')$ comme repr{\'e}sentation de $G(\fp)$ si et seulement si
  $\lambda - \lambda' \in (p-1)X^0(T)$.
\end{pro}

\begin{de}\label{df:reg}
  Un poids de Serre $F$ est dit \emph{r{\'e}gulier} si $F \cong F(\lambda)$ avec $0 \le \langle \lambda, \alpha_i^\vee\rangle < p-1$ pour $i = 1$, $2$.
  On dit de m{\^e}me qu'un tel $\lambda$ est $p$-r{\'e}gulier et on note $X\reg(T) \subset X_1(T)$ l'ensemble des poids $p$-r{\'e}guliers.
\end{de}

En particulier, il y a $p^2(p-1)$ poids de Serre dont $(p-1)^3$ sont r\'eguliers.

Soit $X$ le mod{\`e}le canonique, d{\'e}fini sur $\Q$, de la tour $(X_K)_K$ des vari{\'e}t{\'e}s de Shimura pour $G$ dont le
niveau $K$ parcourt l'ensemble des sous-groupes ouverts compacts de $G(\widehat{\Z})=G(\widehat{\Z}^{p})\times G(\Z_p)$ de la
forme $K=K^p\times G(\Z_p)$ pour un $K^p$ arbitraire de niveau premier {\`a}~$p$.  Soit $f : A\rightarrow X$ la
vari{\'e}t{\'e} ab{\'e}lienne universelle principalement polaris{\'e}e avec structure de niveau premier {\`a}~$p$. En fixant
un point g{\'e}om{\'e}trique $\overline{x}$ de $X$, on obtient une repr{\'e}sentation continue
$\phi_{\overline{x}}:\pi_1(X,\overline{x})\rightarrow \GSp(A_{\overline{x}}[p]^\vee) \cong G(\fp)$ associ\'ee au faisceau
\'etale localement constant $R^1 f_* \Z/p$. Pour toute repr{\'e}sentation $(W,r)$ de $G$ d{\'e}finie sur $\bF_p$, la
compos{\'e}e $r\circ\phi_{\overline{x}}$ fournit par la th{\'e}orie du $\pi_1$ un faisceau {\'e}tale localement constant
$W_X$ sur $X$.

\begin{de}
  On dit que $\rhob$ est \emph{modulaire de poids de Serre $F(\lambda)$} si sa contragr{\'e}diente $\rhob^\vee$ intervient
  comme sous-repr{\'e}sentation du $\Gamma$-module $H^\expbull_{et} (X\times\overline{\Q},F(\lambda)_X)$. On note
  $\cW(\overline\rho)$ l'ensemble de ces poids de Serre, et $\cW\reg(\overline\rho)$ le sous-ensemble de ceux qui sont
  r{\'e}guliers.
\end{de}

Supposons maintenant que $\rhob$ soit mod{\'e}r{\'e}ment ramifi{\'e}e en~$p$.
Pour {\'e}noncer la conjecture, on va d{\'e}finir une repr{\'e}sentation
$V(\rhob|_{I_p})$, de dimension finie sur $\qpb$, du groupe fini $G(\fp)$,
ainsi qu'un op{\'e}rateur $\RR$ envoyant les poids de Serre vers les poids de Serre r{\'e}guliers.

Commen\c cons par l'operateur $\RR$. Soient $\rhot = (2,1;3) \in X(T)$ la demi-somme des racines positives {\`a} translation par $X^0(T)$ pr{\`e}s, et $w_0 = (s_0s_1)^2$ 
l'{\'e}l{\'e}ment le plus long de $W_G$. Si $w \in W_G$ et $\mu \in X(T)$ on note
\[ w\cbull \mu = w(\mu + \rhot)-\rhot. \]
Pour tout $\mu \in X(T)$ on d{\'e}finit $F(\mu)\reg$ de la mani{\`e}re
suivante: on choisit $\mu' \in X\reg(T)$ tel que $\mu - \mu' \in
(p-1)X(T)$ et on pose $F(\mu)\reg = F(\mu')$.
Il est facile de v{\'e}rifier que c'est ind{\'e}pendant du choix de $\mu'$. Finalement on d{\'e}finit
\[ \RR(F(\lambda)) := F(w_0 \cbull (\lambda - p\rhot))\reg. \]

Passons {\`a} $V(\rhob|_{I_p})$. Dans ce qui suit, on utilisera le groupe dual $\widehat G$ (sur $\fp$), identifi\'e avec
$\GSp_4$ par l'isomorphisme $\spin$ (voir p.~\pageref{spin}). Pendant les paragraphes suivants on utilisera le langage
classique pour les groupes alg\'ebriques, donc on \'ecrira $G$ \`a la place de $G(\fpb)$, etc.  On note $\Fr$ le morphisme de
Frobenius sur $G$ ou $\widehat G$.  On a une dualit{\'e} entre les couples $(\T,\theta)$ constitu{\'e}s d'un tore maximal
rationnel $\T \subset G$ et d'un caract{\`e}re $\theta : \T^{\Fr} \to \qpb^\times$ et les couples $(\widehat \T,\sigma)$
constitu{\'e}s d'un tore maximal rationnel $\widehat\T \subset \widehat G$ et d'un {\'e}l{\'e}ment $\sigma \in \widehat\T
^{\Fr}$. De tels couples en dualit{\'e} sont dits \emph{maximalement d{\'e}ploy{\'e}s} si le rang rationnel de~$\widehat\T$
est maximal parmi tous les tores maximaux rationnels de~$\widehat G$ qui contiennent~$\sigma$~\cite{DL}, \S 5.

L'application~$V$ est alors d{\'e}finie par le diagramme suivant~\cite{He}, \S 6.4 (utilisant que $Z(G) \cong \G_m$ est
connexe). Ici l'action de ${\widehat G}^{\Fr}$, resp.\ de ${G}^{\Fr}$, est donn{\'e}e par conjugaison. Les bijections
verticale {\`a} gauche et horizontale en bas d{\'e}pendent du choix d'un g{\'e}n{\'e}rateur du groupe d'inertie
mod{\'e}r{\'e} $I_p^\mathrm{mod}$, mais l'application~$V$ est ind{\'e}pendante de ce choix. La repr{\'e}sentation
$R_\T^\theta$ associ{\'e}e {\`a} un couple $(\T,\theta)$ est celle d{\'e}finie par Deligne--Lusztig~\cite{DL}.
\[ \xymatrix{ 
    {\cong \backslash\bigg\{\begin{array}[h]{@{}c@{}}
        \text{$\tau : I_p \to \widehat G(\fpb)$ mod{\'e}r{\'e}es} \\
        \text{qui s'{\'e}tendent {\`a} $D_p$}
      \end{array} \bigg\} } \ar@{^{(}-->}[r]^V \ar@{<->}[d] &
    {\bigg\{\begin{array}[h]{@{}c@{}}
        \text{repr{\'e}sentations virtuelles} \\
        \text{de $G^{\Fr}$ sur $\qpb$}
      \end{array} \bigg\} / \cong}  \ar@{<-_{)}}[d]^{R_\T^\theta} \\
    {{\widehat G}^{\Fr} \backslash \bigg\{\begin{array}[h]{@{}c@{}}
        \text{couples $(\widehat\T,\sigma)$, $\sigma \in \widehat\T^{\Fr}$} \\
        \text{max.\ deploy{\'e}s}
      \end{array} \bigg\}} \ar@{<->}[r]^{\text{dualit{\'e}\quad\ \ }} & {\bigg\{\begin{array}[h]{@{}c@{}}
        \text{couples $(\T,\theta)$, $\theta : \T^{\Fr} \to \qpb^\times$} \\
        \text{max.\ deploy{\'e}s}
      \end{array} \bigg\} / {G}^{\Fr}} } \]
  Remarquons que $(-1)^{3-\rg_{\fp}\T} V(\tau)$ est toujours une repr{\'e}sentation effective~\cite{DL}, 10.10.
  Les classes de $G^{\Fr}$-conjugaison de tores maximaux rationnels sont param\'etr\'ees par les classes de conjugaison de~$W$ (voir par exemple
  \cite{DL}, 1.14 ou \cite{He}, \S4.1).  En utilisant la preuve de~\cite{DL}, 10.10, on peut v{\'e}rifier que $(-1)^{3-\rg_{\fp}\T} V(\tau)$ est
  l'induite d'un caract{\`e}re de~$B(\fp)$, resp.\ l'induite parabolique d'une repr{\'e}sentation cuspidale de~$M(\fp)$,
  resp.\ l'induite parabolique d'une repr{\'e}sentation cuspidale de~$M_1(\fp)$, resp.\ une repr{\'e}sentation cuspidale
  de~$G(\fp)$ si $\T$ est un tore de type $\{1\}$, resp.\ $\{s_0, s_1s_0s_1\}$, resp.\ $\{s_1,s_0s_1s_0\}$, resp.\
  $\{s_0s_1,s_1s_0\}$ ou $\{w_0\}$.

Pour $\rhob$  mod{\'e}r{\'e}ment ramifi{\'e}e en~$p$ posons $$\cW^?(\rhob|_{I_p}) = \RR(\JH(\overline{V(\rhob|_{I_p})})).$$

\begin{conj}\label{conj:serreweights}
 Soit $\rhob : \Gamma \to \GSp_4(\fpb)$ une repr{\'e}sentation irr{\'e}ductible et motiviquement impaire. Alors,
 
 (i) l'ensemble $\cW(\rhob)$ des poids de Serre de $\rhob$ est non-vide et
son sous-ensemble $\cW\reg(\rhob)$ des poids r{\'e}guliers est contenu dans $\cW^?(\rhob|^{ss}_{I_p})$,

(ii) si de plus $\rhob$ est mod\'er\'ement ramifi\'ee en~$p$, on a
  \[ \cW\reg(\rhob) = \cW^?(\rhob|_{I_p}). \]
\end{conj}

\begin{rk}
  La d{\'e}termination des poids de Serre irr{\'e}guliers semble pour le moment hors d'atteinte. On s'attend n\'eanmoins \`a ce que dans le cas o\`u $\rhob|_{I_p}$
  est mod\'er\'ee et g\'en\'erique, $\cW(\rhob)$ ne contienne aucun poids irr\'egulier.
\end{rk}

\begin{rk}
  L'inclusion pr\'edite dans la partie (i) de la conjecture est analogue \`a la propri\'et\'e de $W_{\mathfrak{p}}(\rho)$ dans
  la conjecture de Buzzard--Diamond--Jarvis pour les repr\'esentations de degr\'e deux d'un corps totalement r\'eel \cite{BDJ}, \S3.2.
\end{rk}

\begin{rk}   Cette conjecture ne n{\'e}cessite pas la condition de $p$-ordinarit{\'e} de $\rhob|_{I_p}$. Cependant dans la d{\'e}termination explicite qui
  suit de l'ensemble
  $\cW^?(\rhob|_{I_p})$, on se limitera au cas $p$-ordinaire. \emph{D{\'e}sormais, on suppose donc que $\rhob$ est $p$-ordinaire.}
\end{rk}

Pour cette d{\'e}termination, la notion d'alc{\^o}ve~\cite{J1}, II.6 sera essentielle. On note $C_i$ ($0\le i \le 3$) l'ensemble
des $(x,y;z) - \rhot \in X(T) \otimes \R$ tel que, respectivement,
\begin{align*}
  C_0 :\ &x > y > 0,\ x+y < p, \\
  C_1 :\ &x+y > p,\ y < x < p, \\
  C_2 :\ &x-y < p < x,\ x+y < 2p, \\
  C_3 :\ &y < p,\ x+y > 2p,\ x-y < p.
\end{align*}
Ce sont les seules alc{\^o}ves pertinentes ici car $X_1(T) = X(T)_+ \cap \bigcup_{i=0}^3 \overline C_i$.

Si $\mu \in X(T) \cong Y(\widehat T)$, on notera $\bar\mu \in Y(\GSp_4)$ le copoids qui correspond \`a $\mu$ par l'isomorphisme $\spin$.
Rappelons que si $\mu = (a,b;c) \in X(T)$, alors $\bar\mu$ est donn{\'e} par
\begin{equation}\label{eq:11}
  \bar\mu : t \mapsto
  \left(\begin{smallmatrix}
    t^{(a+b+c)/2} \\ & t^{(a-b+c)/2} \\ && t^{(-a+b+c)/2} \\ &&&t^{(-a-b+c)/2}
  \end{smallmatrix}\right) \in Y(\GSp_4).
\end{equation}

Pour $\mu \in X(T)$ d{\'e}finissons alors la repr{\'e}sentation mod{\'e}r{\'e}e du groupe d'inertie,
\begin{equation*}
  \tau(1,\mu) = \bar\mu \circ \omega : I_p \to \GSp_4(\fpb),
\end{equation*}
o\`u le \og 1\fg\ signifie qu'il s'agit d'une somme directe des puissances du caract\`ere fondamental de niveau un.

Soit $\rhob : \Gamma \to \GSp_4(\fpb)$ une repr{\'e}sentation $p$-ordinaire et mod{\'e}r{\'e}e, il est facile de voir qu'il existe $\mu \in X(T)$ tel que
$\rhob|_{I_p} \cong \tau(1,\mu)$.

Pour la d{\'e}finition de l'ordre partiel $\uparrow$ sur $X(T)$, qui est plus grossier que $\le$, on renvoie {\`a} \cite{J1}, II.6. 
Mais dans tous cet article il suffit de savoir que la restriction de $\uparrow$ aux poids $\cup \overline{C_i} \cap X(T)$
est l'ordre partiel minimal satisfaisant la condition suivante. Pour chaque $\lambda \in \overline{C_i} \cap X(T)$ ($i = 0$, 1, 2) et $r_i$ la reflexion affine
par rapport au mur entre les alc\^oves $C_i$ et $C_{i+1}$, on a $\lambda \uparrow r_i(\lambda)$. En particulier, pour n'importe quels $0 \le i \le j \le 3$
et pour chaque $\lambda \in \overline{C_i} \cap X(T)$
il existe un unique poids $\lambda' \in \overline{C_j} \cap X(T)$ tel que $\lambda \uparrow \lambda'$.

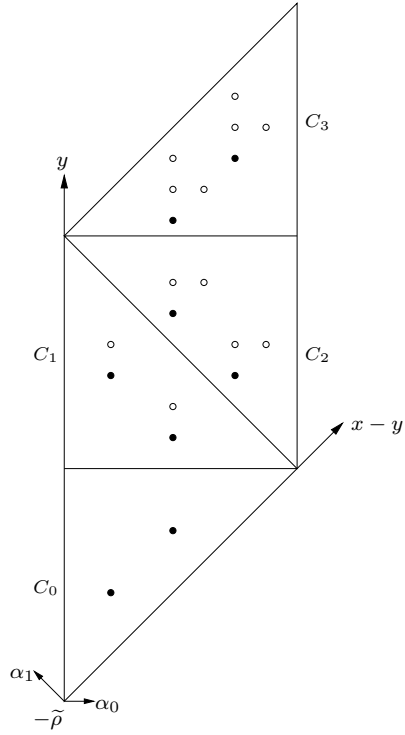
\begin{figure}[t]
  \centering
  \input{20poids.pstex_t}
  \caption{Les vingt poids pr\'edits (situation g\'en\'erique)}\label{fig:20poids}
\end{figure}

\begin{pro}\label{prop:conj_combi} Soit $\tau = \tau(1,\mu)$ une repr{\'e}sentation de $I_p$ comme ci-dessus.
  Alors
  \begin{itemize}
  \item on a $V(\tau) \cong \Ind_{B(\fp)}^{G(\fp)} (\tilde\mu|_{T(\fp)})$ o{\`u} $\tilde\mu : T(\fp) \to \qpb^\times$ est le rel{\`e}vement de Teichm{\"u}ller de
  $\mu : T(\fp) \to \fp^\times$,
\item et l'ensemble $\cW^?(\tau)$ est donn{\'e} par
  \begin{multline}\label{eq:prop}
    \cW^?(\tau) = \{ F(\nu) : \nu \in X\reg(T), \\ \text{$\exists \nu' \uparrow \nu$, $\nu' + \rhot$ dominant et $\tau \cong \tau(1,\nu'+\rhot)$} \}.
  \end{multline}
  \end{itemize}
\end{pro}

\begin{rk}
  Dans la figure~\ref{fig:20poids}, on a indiqu\'e les plus hauts poids $\nu \in X(T)$ des poids de Serre $F(\nu)$ pr\'edits
  pour $\tau$ g\'en\'erique (comparer avec cor.~\ref{cor:conj_combi}). Les points noirs correspondent aux cas
  $\nu = \nu'$ dans~(\ref{eq:prop}), et les points blancs aux autres cas.
\end{rk}

\begin{rk}
L'analogue de ce r\'esultat est valable pour n'importe quelle repr\'esentation $\tau$ mod\'er\'ee et g\'en\'erique qui admet un prolongement \`a $D_p$
(voir~\cite{He}, prop.~6.28).
\end{rk}

\begin{proof}
Le couple $(\T,\theta)$ associ{\'e} {\`a} $\tau$ est $(T,\tilde\mu|_{T(\fp)})$ par d{\'e}finition \cite{He}, \S 6.4. Comme $T \subset B$ on obtient
$V(\tau) \cong \Ind_{B(\fp)}^{G(\fp)} (\tilde\mu|_{T(\fp)})$ \cite{DL}, 8.2.

On va utiliser la formule de Jantzen pour calculer $\JH(\overline{V(\tau)})$ \cite{He}, th.~3.4 dans l'appendice (voir aussi \cite{He}, \S
5.1). Par un calcul {\'e}l{\'e}mentaire on d{\'e}termine d'abord les constantes $\gamma'_{w_1,w_2} \in \Z[X(T)]^{W_G}$ qui
interviennent dans la formule et on v{\'e}rifie qu'elles sont contenues dans $X^0(T)$ si $w_1 = w_2$ ou $(w_1,w_2) \in
\{(\pi,s_1s_0\pi), (s_1s_0s_1\pi,s_1\pi) : \pi \in \langle s_0s_1s_0 \rangle \}$ et sont nulles dans les autres cas.
Rappelons que la d{\'e}finition de $W(\lambda)$, consid{\'e}r{\'e} comme {\'e}l{\'e}ment du groupe de Grothendieck des
repr{\'e}sentations de $G$, s'{\'e}tend {\`a} tout $\lambda \in X(T)$ (pas n{\'e}cessairement dominant) \cite{J1}, II.5.7. Si
$\mu = (x,y;z)$, on obtient alors que $\overline{V(\tau)}$ est {\'e}gale {\`a}
\newcommand{\sss}{\scriptscriptstyle}
\begin{alignat*}{2}
  &W(2(p-1)-x,p-1-y;z+p-1) &&+ W(x+p-1,p-1-y;z) \\[-5pt]
  & \hspace{1.5cm} \sss = W_{3,A} && \hspace{1.5cm} \sss = W_{3,B} \\
  &+ W(y+p-1,x;z+p-1)      &&+W(y+p-1,p-1-x;z) \\[-5pt]
  & \hspace{1.5cm} \sss = W_{2,A} && \hspace{1.5cm} \sss = W_{2,B} \\
  &+ W(p-1-y,x;z+p-1)      &&+W(p-1-y,p-1-x;z) \\[-5pt]
  & \hspace{1.5cm} \sss = W_{1,A} && \hspace{1.5cm} \sss = W_{1,B} \\
  &+ W(p-1-x,y;z+p-1)      &&+W(x,y;z) \\[-5pt]
  & \hspace{1.5cm} \sss = W_{0,A} && \hspace{1.5cm} \sss = W_{0,B} \\
  &+ W(p-2-y,x-1;z+p-1)    &&+ W(p-2-y,p-2-x;z) \\[-5pt]
  & \hspace{1.5cm} \sss = W_{1,A'} && \hspace{1.5cm} \sss = W_{1,B'} \\
  &+ W(p-3-x,y;z+p-1)      &&+ W(x-2,y;z) \\[-5pt]
  & \hspace{1.5cm} \sss = W_{0,A'} && \hspace{1.5cm} \sss = W_{0,B'}
\end{alignat*}
dans le groupe de Grothendieck des repr{\'e}sentations de $G$ sur $\fpb$ (on a donn{\'e} un nom {\`a} chaque terme pour simplifier les
notations dans ce qui suit).

Pour d{\'e}composer ces modules de Weyl on utilise les formules suivantes. Puisque
\begin{equation}
  \label{eq:2}
  W(w\cbull \lambda) = \sgn(w) W(\lambda) \quad \forall w\in W_G
\end{equation}
\cite{J1}, II.5.9(1) on peut se ramener d'abord au cas o{\`u} $\lambda$ est dominant et en fait $p$-restreint. Pour un tel poids $\lambda \in X_1(T)$, on a
\begin{alignat}{2}
\label{eq:1}  W(\lambda) &= F(\lambda) + F(s_1s_0s_1 \cbull \lambda + p(2,2;0)) && \quad\text{si $\lambda \in C_3$}, \\
W(\lambda) &= F(\lambda) + F(s_0s_1s_0 \cbull \lambda + p(2,0;0)) && \quad\text{si $\lambda \in C_2$}, \\
W(\lambda) &= F(\lambda) + F(s_1s_0s_1 \cbull \lambda + p(1,1;0)) && \quad\text{si $\lambda \in C_1$}, \\
\label{eq:10} W(\lambda) &= F(\lambda) && \quad\text{si $\lambda \in C_0$}.
\end{alignat}
Remarquons que si $\lambda \in C_i$ avec $1 \le i \le 3$, les composants de $W(\lambda)$ sont dans $C_i$ et $C_{i-1}$.
Si $\lambda \in X_1(T) \setminus \cup_{i=0}^3 C_i$, $W(\lambda)$ est irr{\'e}ductible sauf si $\langle \lambda+\rhot,\alpha_0\dual \rangle =
p$ et $p/2 < \langle \lambda+\rhot,\alpha_1\dual \rangle < p$, ou $p = 2$ et $\lambda$ est de la forme $(1,1;*)$. Dans ces cas
exceptionnels $W(\lambda)$ se d{\'e}compose selon~\eqref{eq:1} \cite{J0}, \S 7.
 (Bien que les r{\'e}sultats soient enonc{\'e}s pour $\Sp_4$ dans cet
article de Jantzen, ils sont encore valables pour $\GSp_4$ parce que son groupe d{\'e}riv{\'e} est $\Sp_4$.)

Nous allons maintenant {\'e}tudier le membre de droite de l'{\'e}galit{\'e}~\eqref{eq:prop}. En utilisant l'{\'e}quivalence
\begin{equation}
  \label{eq:4}
  \tau(1,\mu) \cong \tau(1,\mu') \Longleftrightarrow \mu' \in W_G\,\mu + (p-1)X(T), 
\end{equation}
nous pouvons supposer dor{\'e}navant que
\begin{equation}
  \label{eq:3}
  x-y \ge 0,\quad y \ge 0,\quad x \le (p-1)/2.
\end{equation}
Observons que pour $G = \GSp_4$, pour tout poids $\nu$ $p$-r{\'e}gulier et pour tout poids $\nu' \in X(T)$ tel que $\nu' + \rhot$ soit dominant et $\nu'
\uparrow \nu$, on a 
\[ 0 \le \langle \nu'+\rhot, \alpha_i\dual \rangle < p \quad \forall i. \]
(Cette propri{\'e}t{\'e} est particuli{\`e}re {\`a} la g{\'e}ometrie des alc{\^o}ves de $\GSp_4$.) Il r{\'e}sulte donc facilement de~\eqref{eq:4} que les $\nu'$ dans le membre de droite
de~\eqref{eq:prop} sont pr{\'e}cis{\'e}ment\label{eq:nu_prime}
\begin{align*}
  \nu'_{0,A} &= (x,y;z)-\rhot, & \nu'_{0,B} &= (p-1-x,y;z+p-1)-\rhot, \\
  \nu'_{1,A} &= (p-1-y,p-1-x;z)-\rhot, & \nu'_{1,B} &= (p-1-y,x;z+p-1)-\rhot, \\
  \nu'_{2,A} &= (y+p-1,p-1-x;z)-\rhot, & \nu'_{2,B} &= (y+p-1,x;z+p-1)-\rhot, \\
  \nu'_{3,A} &= (x+p-1,p-1-y;z)-\rhot, & \nu'_{3,B} &= (2p-2-x,p-1-y;z+p-1)-\rhot,
\end{align*}
{\`a} $(p-1)X^0(T)$ pr{\`e}s. (Cette liste est valable m\^eme si $x = y$ ou $y = 0$.)

Commen\c cons {\`a} comparer les deux membres de l'\'equation~\eqref{eq:prop}.
On dira qu'un $\nu' \in X(T)$ \emph{explique} une repr{\'e}sentation $W$ de $G(\fp)$ sur $\fpb$ si
\begin{equation}
  \label{eq:5}
  \{ F(\nu) : \nu \in X\reg(T),\ \nu' \uparrow \nu \} = \RR(\JH(W)).
\end{equation}
Alors
\begin{alignat}{2}
\label{eq:6}  &\text{$\nu'_{0,X}$ explique $W_{3,X}\oplus W_{1,X'}$} &\quad& \text{si $\nu'_{0,X} \in C_0$,} \\
\label{eq:7}  &\text{$\nu'_{1,X}$ explique $W_{2,X}\oplus W_{0,X'}$} && \text{si $\nu'_{1,X} \in C_1$,} \\
\label{eq:8}  &\text{$\nu'_{2,X}$ explique $W_{1,X}$} && \text{si $\nu'_{2,X} \in C_2$,} \\
\label{eq:9}  &\text{$\nu'_{3,X}$ explique $W_{0,X}$} && \text{si $\nu'_{3,X} \in C_3$,}
\end{alignat}
o{\`u} \og X\fg\  signifie soit \og A\fg\  soit \og B\fg. En particulier, si ces huit conditions sont satisfaites simultan{\'e}ment, la proposition est
prouv{\'e}e.  Dans une premi{\`e}re {\'e}tape on va {\'e}tendre la preuve aux cas o{\`u} $\nu'_{i,X} \in \overline C_i$ pour chaque $0 \le i \le 3$.
Avec cette hypoth{\`e}se, le m{\^e}me raisonnement que pr{\'e}c{\'e}demment s'applique pourvu que $\langle \nu'_{i,X}+\rhot ,\alpha_j\dual \rangle \ne 0$ pour
chaque $i$, $j$.  Si $\langle \nu'_{1,X}+\rhot, \alpha_0\dual \rangle = 0$, l'\'equation \eqref{eq:7} n'est plus valable parce qu'un poids de Serre manque
{\`a} la gauche de l'{\'e}quation~\eqref{eq:5}; mais \c ca ne pose pas un probl{\`e}me car c'est $F(\nu'_{3,Y})$ (o{\`u} \og Y\fg\  est ou \og A\fg\  ou
\og B\fg\  mais pas {\'e}gale {\`a} \og X\fg).  Un argument similaire s'applique si $\langle \nu'_{0,X}+\rhot, \alpha_0\dual \rangle = 0$, auquel cas il manquent en g\'en\'eral deux poids au membre de gauche, ou si $\langle \nu'_{0,X}+\rhot, \alpha_1\dual \rangle = 0$.  (Les poids qui manquent au membre de gauche
de~\eqref{eq:5} dans tous ces cas correspondent aux \og discontinuit{\'e}s\fg\  de $\RR$, i.e., aux poids $\lambda$ o{\`u} $w_0\cbull (\lambda-p\rhot) \not\in X\reg(T)$;
il y a alors toujours un autre $\nu'$ qui explique le poids qui manque parce que $\tau(1,-)$ est constant sur $\nu'+\rhot+(p-1)X(T)$.)

Remarquons ensuite que~\eqref{eq:9} est vraie m{\^e}me si $\nu'_{3,X} \in \overline C_2$ (m{\^e}me argument que pour $\nu'_{2,X}$).

Il reste \`a traiter les cas o{\`u} $x = y$ ou $y = 0$. Si $y = 0$ et $x \ne y$, tout marche comme
$\nu'_{2,X} = \nu'_{1,X} \in \overline C_1$ et $W_{2,X} = W_{1,X}$.

Si $x = y$ et $0 < y < (p-1)/2$, les {\'e}quations~\eqref{eq:6}--\eqref{eq:9}, sauf~\eqref{eq:7} quand $X = B$, sont satisfaites.
Comme $\nu'_{1,B} = \nu'_{0,B} \in C_0$ et $W_{2,B} = W_{3,B}$ il suffit de montrer que $W_{0,B'}$ ne contribue pas {\`a} $\RR(\JH(\overline{V(\rhob|_{I_p})}))$,
c'est-{\`a}-dire ne change que les multiplicit{\'e}s des composants. Soit $F_{1,B'}$ (resp., $F_{3,A}$) le composant de $W_{1,B'}$ (resp., $W_{3,A}$) avec
plus haut poids dans $\overline C_0$ (resp.\  $\overline C_2$). Alors, en effet, $W_{0,B'} = -F_{1,B'}$ et $\RR(F_{1,B'}) = \RR(F_{3,A})$.
Si $x = y$ et $y = (p-1)/2$, on traite en plus le cas~\eqref{eq:7} pour $X = A$ de la m{\^e}me mani{\`e}re.

Supposons finalement que $x = y = 0$. Observons que $\nu'_{2,A} = \nu'_{3,A} = \nu'_{1,A} \in C_1$, $W_{1,A} = W_{0,A} = W_{2,A}$ et
$\nu'_{2,B} = \nu'_{1,B} = \nu'_{0,B} \in C_0$, $W_{1,B} = W_{2,B} = W_{3,B}$. Les {\'e}quations~\eqref{eq:6}, \eqref{eq:7} si $X = A$
et~\eqref{eq:9} si $X = B$ sont satisfaites. Il ne reste qu'{\`a} constater que $W_{1,B'} = 0$. On v{\'e}rifie ais{\'e}ment que l'argument de ce paragraphe est encore valable si $p = 2$ (et c'est l'unique cas qui peut arriver si $p = 2$).
\end{proof}

\begin{lem}\label{lm:twisting} On a
  \[ \cW^?(\tau \otimes \omega^c) \cong \cW^?(\tau) \otimes \nu_\simi^c \qquad \forall c \in \Z, \]
o\`u $\nu_\simi : G \to \G_m$ est
le facteur de similitudes de $G$. De m\^eme,
 \[ \cW(\rhob \otimes \omega^c) \cong \cW(\rhob) \otimes \nu_\simi^c \qquad \forall c \in \Z.  \]
\end{lem}

\begin{proof}
  Pour la premi\`ere assertion, on observe que $\nu_\simi|_{T} = (0,0;2) \in X(T) \cong Y(\widehat T)$ correspond par l'isomorphisme $\spin$
  au cocaract\`ere central $\zeta : z\mapsto {\rm diag}(z,z,z,z)$ de $\G_m$ vers $G$ (voir~\eqref{eq:11}).
  Puisque $\nu_\simi$ et $\zeta$ sont invariants par conjugaison et $R^{\theta
    \widetilde \nu_\simi}_\T \cong R^\theta_\T \otimes \widetilde \nu_\simi$ (\cite{DL}, 1.27), on a $V(\tau \otimes \omega^c) \cong
  V(\tau) \otimes \widetilde \nu_\simi^c$.  Finalement, on utilise que $\RR(F \otimes \nu_\simi) \cong \RR(F) \otimes \nu_\simi$.
  
  Pour la deuxi\`eme assertion, rappelons la construction du syst\`eme local $F_X$ associ\'e \`a une repr\'esentation
  $(F,\rho_F)$ de $G(\Z/p\Z)$. Soit $f : A \to X_\Q$ le sch\'ema ab\'elien universel et soit $\phi_{\overline x} :
  \pi_1(X_\Q,\overline{x})\rightarrow \GSp(A_{\overline x}[p]\dual) \cong G(\Z/p\Z)$ la repr\'esentation du groupe fondamental
  de $X\times\Q$ (pour un point g\'eom\'etrique fix\'e), associ\'ee au faisceau \'etale localement constant $R^1 f_* \Z/p$.
  Le syst\`eme local $F_X$ est associ\'e \`a la repr\'esentation $\rho_F\circ\phi_{\overline x}$ du groupe fondamental.  Soit
  $\rho_{\overline x}:\pi_1(X_\Q,\overline{x})\rightarrow \GSp(A_{\overline x}[p]) \cong G(\Z/p\Z)$ la repr\'esentation
  associ\'ee au rev\^etement \'etale $A[p]\rightarrow X_\Q$ des points de $p$-torsion.  Comme la forme symplectique sur
  $A_{\overline x}[p]$ est donn\'ee par l'accouplement de Weil, on a pour tout $P,Q\in A_{\overline x}[p]$: $\langle
  P^\sigma,Q^\sigma\rangle= \langle P,Q\rangle^{\nu_\simi\circ\rho_{\overline x}(\sigma)}=\langle P,Q\rangle^{\omega(\sigma)}$. On
  trouve donc $\nu_\simi\circ\rho_{\overline x}=\omega$ ce qui implique $\nu_\simi \circ \phi_{\overline x} = \omega^{-1}$.  Il s'ensuit
  que
$$H^\expbull_{et}(X\times\overline{\Q},F(\lambda)_X\otimes\nu_{\simi,X}^c)= H^\expbull_{et}(X\times\overline{\Q},F(\lambda)_X)\otimes\omega^{-c}.$$
Le r\'esultat en d\'ecoule.
\end{proof}

La preuve de la proposition~\ref{prop:conj_combi} donne le corollaire suivant, qui fournit une description explicite de
$\cW^?(\tau)$ dans la plupart des cas, \`a torsion pr\`es.

\begin{cor}\label{cor:conj_combi}
Supposons que $\tau : I_p \to \GSp_4(\fpb)$ est donn\'ee par
\[ \tau \sim \left(\begin{array}{cccc}\omega^{k+\ell-3}&0&0&0\\0&\omega^{k-1}&0&0\\0&0&\omega^{\ell-2}&0\\0&0&0&1\end{array}\right)
  \sim \tau(1,(x,y;x+y)), \]
o\`u $k > \ell > 3$, $k+\ell < p+1$ et $x = k-1$, $y = \ell-2$. Alors $\cW^?(\tau)$ consiste
en les 20 poids de Serre $F(\nu)$ avec $\nu+\rhot \in X\reg (T)+\rhot$ parcourant:
\[\begin{array}{l|ll}
C_0 & (x,y;x+y), & (p-1-x,y;x+y+p-1), \\
\hline
C_1 & (p-1-y,p-1-x;x+y), & (p-1-y,x;x+y+p-1), \\
C_0 \to C_1 & (p-y,p-x;x+y), & (p-y,x+1;x+y+p-1), \\
\hline
C_2 & (y+p-1,p-1-x;x+y), & (y+p-1,x;x+y+p-1), \\
C_1 \to C_2 & (y+p+1,p-1-x;x+y), & (y+p+1,x;x+y+p-1), \\ 
C_0 \to C_2 & (y+p,p-x;x+y), & (y+p,x+1;x+y+p-1), \\
\hline
C_3 & (x+p-1,p-1-y;x+y), & (2p-2-x,p-1-y;x+y+p-1), \\
C_2 \to C_3 & (x+p+1,p+1-y;x+y), & (2p-x,p+1-y;x+y+p-1), \\ 
C_1 \to C_3 & (x+p+1,p-1-y;x+y), & (2p-x,p-1-y;x+y+p-1), \\
C_0 \to C_3 & (x+p,p-y;x+y), & (2p-1-x,p-y;x+y+p-1).
\end{array}\]
Ici la notation \og $C_j \to C_i$\fg\ pour $j < i$, resp.\ \og $C_i$\fg, veut dire que les $\nu$ correspondants sont
situ\'es dans l'alc\^ove $C_i$ et qu'ils sont obtenus en commen\c cant par un $\nu' \in C_j$, resp.\ par $\nu' = \nu$, dans
la description de $\cW^?(\tau)$ de la prop.~\ref{prop:conj_combi}. Il y a une seule exception: si $k-\ell = 1$ ou $k + \ell =
p$, un $\nu$ de la rang\'ee \og $C_3$\fg\ se trouve sur la fronti\`ere entre $C_2$ et $C_3$.
\end{cor}

\begin{rk}
  Supposons que $\rhob = \rhob_{f,p}$ avec $f$ de poids $(k, \ell)$ et $k > \ell > 3$ et $k+\ell < p+1$ telle que $\rhob|_{I_p}$ est $p$-ordinaire
  mod\'er\'ee. Notons $\lambda = (a,b;a+b) = (x,y;x+y)-\rhot$ le \og poids fondamental\fg. Alors les huit poids $\nu$ des rang\'ees \og $C_i$\fg\ se
  r{\'e}partissent en deux sous-ensembles. Le premier est le sous-ensemble des poids $p$-r\'eguliers parmi $w\cbull\lambda+\Z(p-1,p-1)$ o{\`u} $w\in
  W^M$ parcourt l'ensemble des repr{\'e}sentants de Kostant dans $W_G$, et le second est le sous-ensemble des poids $p$-r\'eguliers parmi $w\cbull
  \lambda+\Z(p-1,0)$ pour $w\in W_G\setminus W^M$. Les poids de Serre correspondants s'appellent les poids compagnons de~$f$ (ou de son poids
  fondamental).

  Remarquons que sous la condition plus faible $k \ge \ell \ge 3$ et $k + \ell - 3 < p-1$, qui provient de la th\'eorie de Fontaine--Laffaille, on a toujours ces
  huit poids compagnons du poids fondamental. Ils sont tous $p$-restreints mais peuvent se situer sur des murs ou sur des alc\^oves 
  contigu\"es \`a
  celles de la liste ci-dessus. On y reviendra dans la prochaine section.
\end{rk}

En ce qui concerne les douze poids dans les rang\'ees \og $C_j \to C_i$\fg, on peut justifier leur pr\'esence dans la liste
par la proposition suivante. On dira qu'une repr\'esentation $\rho_p : D_p \to \GSp_4(\qpb)$ est \emph{cristalline \`a poids
  de Hodge--Tate} $\xi \in Y(T)$ si $\jmath \circ \rho_p$ est cristalline \`a poids de Hodge--Tate $\jmath \circ \xi$
(qui s'identifie \`a un multiensemble de quatre entiers), o\`u $\jmath : \GSp_4 \into \GL_4$ est l'inclusion canonique.

\begin{pro}\label{prop:cryst_lifts}
  Supposons que la repr\'esentation locale $\rhob_p : D_p \to \GSp_4(\fpb)$ est totalement d\'ecompos\'ee.
  Alors pour chaque $\mu \in X\reg(T)+\rhot$ tel que $F(\mu-\rhot) \in \cW^?(\rhob_p\vert_{I_p})$ il existe
  une repr\'esentation cristalline $\rho_p : D_p \to \GSp_4(\qpb)$ \`a poids de Hodge--Tate $\bar\mu \in Y(T)$ qui rel\`eve $\rhob_p$.
\end{pro}

\begin{rk}\label{rk:cryst_lifts}\ 

(i) Cet \'enonc\'e fournit un indice de nature locale pour l'existence pour chaque poids $\mu\in\cW^?(\rhob_p\vert_{I_p})$, d'une repr\'esentation automorphe cuspidale $\pi$ telle que l'on ait (globalement): $\overline{\rho}_\pi=\overline{\rho}$. En effet, la repr\'esentation galoisienne associ\'ee \`a une repr\'esentation cuspidale cohomologique de poids $\mu-\rhot$
  (donc dont la composante \`a l'infini est de param\`etre de Harish--Chandra $\mu+(0,0;3)$) et de niveau premier \`a~$p$ est cristalline de poids de Hodge--Tate $\bar\mu$.
  
  (ii) L'analogue de la proposition pour ${\GL_n}_{/\Q}$ est pr\'edit par une conjecture de Gee~\cite{Gee}, \S 4.3 m\^eme
  sans supposer que $\rhob|_{I_p}$ est mod\'er\'ee et que le poids de Serre est r\'egulier. La g\'en\'eralisation de cette
  conjecture \`a beaucoup d'autres groupes r\'eductifs (incluyant ${\GSp_4}_{/\Q}$) sera le sujet de \cite{GHS}.
\end{rk}

\begin{proof}
Pour $\alpha \in \fpb\s$ notons $\ta$ le rel\`evement de Teichm\"uller de $\alpha$.

La repr\'esentation $\rhob_p$ \'etant totalement d\'ecompos\'ee, 
il existe des entiers $x \ge y \ge 0$ et $z$ avec $x+y \le p-1$ et
des $\alpha$, $\beta$, $\gamma$, $\delta \in \fpb\s$ avec $\alpha\delta = \beta\gamma$ tel que
\[ \rhob_p \sim \left(\begin{array}{cccc}\omega^{x+y}\nr(\alpha)&0&0&0\\0&\omega^{x}\nr(\beta)&0&0\\
    0&0&\omega^{y}\nr(\gamma)&0\\0&0&0&\nr(\delta)\end{array}\right) \otimes \omega^z. \]
(On utilise les \'equations~\eqref{eq:4}, \eqref{eq:3}.) Notons d'abord que l'\'enonc\'e d\'epend seulement de $\mu$ modulo $(p-1)X^0(T)$,
et donc de $F(\mu-\rhot)$,
puisqu'on peut tordre $\rho_p$ par $\epsilon^{n(p-1)}$ avec $n \in \Z$ sans changer sa r\'eduction. \'Egalement, apr\`es torsion nous pouvons
supposer que $z = 0$.

Commen\c cons par exhiber un rel\`evement cristallin \`a poids de Hodge--Tate $\bar\mu$ pour chaque \'el\'ement $\mu$ de la liste du
cor.~\ref{cor:conj_combi} (n'importe que la restriction sur $(x,y)$ est plus faible ici). Puis on v\'erifiera que l'on a obtenu ainsi (au moins) un rel\`evement
pour chaque poids de Serre dans $\cW^?(\rhob_p\vert_{I_p})$. En fait on peut d\'eduire facilement de la preuve qu'aucun des rel\`evements que l'on
a d\'ecrit ne correspond \`a un poids de Serre r\'egulier qui n'est pas pr\'edit.

Consid\'erons les quatre expressions suivantes pour $\rhob_p$:
\begin{align}
  &\omega^{x+y}\nr(\alpha) \oplus \omega^{x}\nr(\beta)\oplus \omega^{y}\nr(\gamma)\oplus \nr(\delta),\label{eq:12} \\
  &\omega^{p-1}\nr(\delta)\oplus \omega^{x}\nr(\beta)\oplus \omega^{y}\nr(\gamma)\oplus \omega^{x+y-p+1}\nr(\alpha), \label{eq:13}\\
  &\omega^{y+p-1}\nr(\gamma)\oplus \omega^{x+y}\nr(\alpha)\oplus \nr(\delta)\oplus \omega^{x-p+1}\nr(\beta), \label{eq:14}\\
  &\omega^{x+p-1}\nr(\beta)\oplus \omega^{x+y}\nr(\alpha)\oplus \nr(\delta)\oplus \omega^{y-p+1}\nr(\gamma). \label{eq:15}
\end{align}
Dans chaque cas il y a un rel\`evement cristallin \'evident obtenu en rempla\c cant $\omega$ par $\epsilon$ et $\alpha$, \ldots, $\delta$ 
par leurs rel\`evements
de Teichm\"uller. Les poids $\mu$ correspondants sont ceux des rang\'ees
\og $C_i$\fg\ qui sont situ\'es \`a la gauche.

Regardons l'\'equation~\eqref{eq:15}. Puisque $2p-2+x-y \ge 2p-2$ et $x+p-1 \ge p-1$ le sous-lemme ci-dessous montre qu'il existe des repr\'esentations
cristallines~$\sigma_i$ ($0 \le i \le 2$) telles que :
\begin{center}
\begin{tabular}{c|c|c}
 & poids de HT & r\'eduction \\
\hline 
$\sigma_0$ & \topspace{11pt}$\{x+p,y-p\}$ & $\omega^{x+p-1}\nr(\beta)\oplus \omega^{y-p+1}\nr(\gamma)$ \\
$\sigma_1$ & $\{x+p,-1\}$ & $\omega^{x+p-1}\nr(\beta)\oplus \nr(\delta)$ \\
$\sigma_2$ & $\{x+p+1,y-p-1\}$ & $\omega^{x+p-1}\nr(\beta)\oplus \omega^{y-p+1}\nr(\gamma)$ \\
\end{tabular}
\end{center}
Pour $i = 0$, $2$ on obtient le rel\`evement cristallin $\epsilon^{x+y}\nr(\ta) \oplus \sigma_i \oplus
\epsilon^{-x-y}\nr(\ta^{-1})\det\sigma_i$ dont l'image est contenue dans le Levi~$M_1$ de Klingen. \'Egalement on obtient le
rel\`evement cristallin $\sigma_1 \oplus \epsilon^{x+y}\nr(\ta\td) \cdot s \cdot {}^t\sigma_1^{-1}\cdot s$ dont l'image est
contenue dans le Levi~$M$ de Siegel. (Voir la section~\ref{sec:notations} pour les notations.) Le m\^eme argument s'applique
aux expressions \eqref{eq:14} pour $i \le 1$ et \eqref{eq:13} pour $i = 0$. Les poids~$\mu$ correspondants
sont ceux des rang\'ees \og $C_i \to C_j$\fg\ qui sont situ\'es \`a la gauche.

Ainsi on a obtenu dix rel\`evements cristallins qui correspondent aux poids~$\mu$ \`a gauche dans le tableau. Comme $\rhob_p$ peut aussi \^etre
\'ecrite sous la forme
\begin{equation*}
   \rhob_p \sim \left(\begin{array}{cccc}\omega^{x'+y'}\nr(\gamma)&0&0&0\\0&\omega^{x'}\nr(\delta)&0&0\\
       0&0&\omega^{y'}\nr(\alpha)&0\\0&0&0&\nr(\beta)\end{array}\right) \otimes \omega^{p-1-x'},
\end{equation*}
avec $(x',y') = (p-1-x,y)$ satisfaisant les m\^emes in\'egalit\'es que $(x,y)$, on obtient dix rel\`evements cristallins correspondant aux poids $\mu$ \`a droite.

D\'emontrons maintenant que l'on a exhib\'e suffisament de rel\`evements. Par la preuve de la prop.~\ref{prop:conj_combi},
les poids de Serre pr\'edits sont pr\'ecis\'ement les $F(\nu)$ o\`u $\nu'_{i,X} \uparrow \nu$ avec
$\nu \in X\reg(T)$. (Les $\nu'_{i,X}$ ont \'et\'e d\'efinis dans la d\'emonstration de cette proposition.)  La
sym\'etrie $(x,y) \leftrightarrow (x',y')$ permet de nous ramener \`a un seul $X \in \{A,B\}$ pour chaque $0 \le i \le 3$.

Commen\c cons par $i = 0$, et prenons $\nu'_{0,A}$. Comme on a $\nu'_{0,A} \in \overline C_0$, les poids $\nu+\rhot \in
X\reg(T)+\rhot$ pour lesquels $\nu'_{0,A} \uparrow \nu$ sont pr\'ecis\'ement les poids de $X\reg(T)+\rhot$ qui
figurent dans les quatre rang\'ees \og $C_0 (\to C_i)$\fg\ de gauche dans la liste. On a d\'ej\`a d\'ecrit un rel\`evement
cristallin dans chacun de ces cas.

Pour $i = 1$ prenons $\nu'_{1,B}$. Quand $x > y$, on a $\nu'_{1,B} \in \overline C_1$ et tout marche comme dans le
premier cas en utilisant les poids dans les trois rang\'ees \og $C_1 (\to C_i)$\fg\ \`a droite. Quand $x = y$
on a $\nu'_{1,B} = \nu'_{0,B}$, donc ce cas a d\'ej\`a \'et\'e
discut\'e (\`a cause de la sym\'etrie entre les cas \og A\fg\ et \og B\fg).

Pour $i = 2$ prenons $\nu'_{2,B}$. Quand $y > 0$, on a $\nu'_{2,B} \in \overline C_2$ et tout marche comme dans le
premier cas en utilisant les poids des deux rang\'ees \og $C_2 (\to C_i)$\fg\ \`a droite. Quand $y = 0$
on a $\nu'_{2,B} = \nu'_{1,B}$, donc ce cas a d\'ej\`a \'et\'e discut\'e.

Finalement pour $i = 3$ prenons $\nu'_{3,A}$. Quand $x > y+1$ on a $\nu'_{3,A} \in \overline C_3$ et tout marche comme
dans le premier cas en utilisant le poids de la rang\'ee \og $C_3$\fg\ \`a gauche. Quand $x = y+1$ on est ramen\'e au cas $i
= 0$ car $\nu'_{3,A}$ est \'egale au poids dans la rang\'ee \og $C_0 \to C_2$\fg\ \`a gauche. Quand $x = y$
on a $\nu'_{3,A} = \nu'_{2,A}$, donc ce cas a d\'ej\`a \'et\'e
discut\'e.

\begin{sublem} Soient $\alpha$, $\beta \in \fpb\s$.

\begin{enumerate}
\item Supposons que $k \ge p+2$. Alors il existe une repr\'esentation cristalline de $D_p$ de dimension~$2$ \`a poids de Hodge--Tate $\{0, k-1\}$
et de r\'eduction $\omega^{k-2}\nr(\alpha) \oplus \omega \nr(\beta)$.
\item Supposons que $k \ge 2p+3$. Alors il existe une repr\'esentation cristalline de $D_p$ de dimension~$2$ \`a poids de Hodge--Tate $\{0, k-1\}$
et de r\'eduction $\omega^{k-3}\nr(\alpha) \oplus \omega^2\nr(\beta)$.
\end{enumerate}
\end{sublem}

Pour d\'emontrer le sous-lemme, fixons d'abord un isomorphisme $\C \cong \qpb$. Soit $K$ un corps de nombres quadratique
imaginaire qui est d\'eploy\'e en~$p$, et fixons un plongement $K \to \C$.  Soit $\ell \ge 1$ un entier. Nous commen\c cons
par construire une forme parabolique de type CM, $f$, de niveau premier \`a~$p$ et de poids $\ell$ telle que $\overline{\rho_f}|_{D_p} \sim
\omega^{\ell-1} \nr(\alpha) \oplus \nr(\beta)$.

Notons $\p$ et $\overline\p$ les id\'eaux premiers divisant~$p$, tel que $\p$ correspond au plongement $K \to \C\cong \qpb$.
Supposons donn\'e un caract\`ere de Hecke $\chi : K\s \backslash \A_K\s \to \C\s$ tel que (a) le
conducteur $\f$ de $\chi$ est premier \`a~$p$, (b) $\chi(z) = z^{1-\ell}$ pour $z \in \C \cong K_\infty$, (c)
$\overline{\chi(\p)p^{1-\ell}} = \alpha$ et $\overline{\chi(\overline\p)} = \beta$, et (d) $\chi \ne \chi \circ c$, o\`u $c$
est l'unique \'el\'ement non-trivial de $\Gal(K/\Q)$.  Alors la forme de type CM associ\'ee (\cite{Mi}, th.~4.8.2),
\begin{equation*}
  f(z) = \sum_{\substack{\a \lhd \cO_K \\ (\a,\f) = 1}} \chi(\a)q^{N(\a)},
\end{equation*}
est la forme cherch\'ee.

On va m\^eme construire pour n'importe quelles racines de l'unit\'e $a$, $b$ un caract\`ere $\chi$ satisfaisant
(a), (b), (d) et (c$'$) $\chi(\p)p^{1-\ell} = a$, $\chi(\pb) = b$. En tordant par un caract\`ere satisfaisant (a) et (b) on
est r\'eduit au cas $\ell = 1$ (le seul cas d'ailleurs o\`u la condition (d) n'est pas automatique).
Puisqu'il existe des caract\`eres de Dirichlet (c'est-\`a-dire sur~$\Q$) de conducteur premier \`a~$p$ de n'importe quelle
valeur d'ordre fini en~$p$, on peut remplacer la condition (c$'$) par (c$''$): $\chi(\p)/\chi(\overline\p) = d := a/b$.

\newcommand{\xp}{x^{(p)}}

Soit $n \ge 1$ minimal tel que $(\p/\overline\p)^n = x\cO_K$ est principal. En consid\'erant les extensions possibles d'un caract\`ere
$K\s\backslash K\s \widehat{\cO_K}\s K_\infty\s \to \C\s$ \`a $K\s \backslash \A_K\s$, on peut remplacer (c$''$) par
(c$'''$): $\chi(\xp) = e := d^{-n}$ o\`u on note $\xp \in \widehat{\cO_K}\s$ l'id\`ele qui est triviale en~$p$ et~$\infty$ et qui co\"incide avec~$x$
dehors.

Alors il suffit de construire un caract\`ere $\eta : (\cO_K/\f)\s \to \C\s$ tel que $\eta(\xp) = e$ et tel que $\eta$ soit
trivial sur l'image de $\cO_K\s$.  En rempla\c cant $e$ par une 12\`eme racine de $e$ et en prenant la 12\`eme puissance de
$\eta$ \`a la fin, on peut abandonner la condition que $\eta|_{\cO_K\s}$ soit triviale (puisque $\cO_K\s$ est fini d'ordre
divisant~12). Sans perte de g\'en\'eralit\'e on peut supposer que l'ordre de $e$ est de la forme $q^r$ o\`u $q$ est premier.
Il suffit donc de trouver un id\'eal~$\f$ premier \`a~$p$ tel que l'ordre de $\xp$ dans $(\cO_K/\f)\s$ est un multiple de
$q^r$. Soit~$h$ le nombre de classes de~$K$. On va montrer qu'il existe un entier $i \ge r$ et un id\'eal $\f$ premier \`a
$p$ tel que $\f \nmid x^{12hq^{r-1}}-1$ mais $\f \mid x^{12hq^i}-1$. Choisissons $\pi \in \cO_K$ tel que $\p^h = \pi \cO_K$.
Alors il suffit de montrer que les normes des id\'eaux $(\pi^{12nq^i} - {\overline\pi}^{12nq^i})\cO_K$ (qui sont premiers
\`a~$p$) ne sont pas born\'ees. C'est un argument \'el\'ementaire utilisant que $\pi/\overline\pi$ n'est pas une racine de
l'unit\'e.

Jusqu'ici on a construit un caract\`ere $\chi$ satisfaisant (a), (b) et (c). Il suffit maintenant de construire un
caract\`ere $\chi'$ satisfaisant (a), (b), (c) et (d) dans le cas $a = b = 1$ (et toujours $\ell = 1$).  Comme ci-dessus on
r\'eduit \`a construire un caract\`ere $\eta : (\cO_K/\f)\s \to \C\s$ tel que $\eta(\xp) = 1$ et tel que $\eta \ne \eta \circ
c$. Soient $q\nmid 6$ et $q_1 \ne q_2$ des nombres premiers tels que $q_i \equiv 1\pmod
{q\operatorname{disc}(K)}$.  Alors les $q_i$ sont d\'ecompos\'es dans $K/\Q$. Choisissons $\q_i \mid q_i$ des id\'eaux premiers
de $\cO_K$ et posons $\f = \q_1\q_2$. Puisque pour chaque \'el\'ement d'un espace vectoriel de dimension deux il
existe un \'el\'ement non-trivial de l'espace dual qui l'annule, il existe $\eta : (\cO_K/\f)\s \to \C\s$ d'ordre exactement
$q$ tel que $\eta(\xp) = 1$ ; \'evidemment on a $\eta \ne \eta \circ c$.

Cela termine la construction du caract\`ere de Hecke $\chi$.

La forme parabolique $\bar f$ sur $\fpb$, la r\'eduction de $f$ modulo~$p$, satisfait $\rho_{\bar f}|_{D_p} \sim
\omega^{\ell-1}\nr(\alpha) \oplus \nr(\beta)$.  Alors la forme $\overline g = \theta \bar f$ est propre et parabolique
de m\^eme niveau, de poids $\ell+p+1$ et telle que $\rho_{\overline g} \sim \rho_{\bar f} \otimes \omega$ \cite{G}, \S 4. Par
le lemme de Deligne--Serre, il existe une forme parabolique et propre $g$ en caract\'eristique nulle de m\^eme poids et
niveau et tel que $\overline{\rho_g}|_{D_p} \sim \rho_{\overline g}|_{D_p} \sim \omega^\ell \nr(\alpha) \oplus \omega \nr(\beta)$ (notons que le poids est
maintenant au moins deux). Puisque la forme $g$ est de niveau premier \`a~$p$ et de poids $\ell+p+1$, on sait que
$\rho_g|_{D_p}$ est cristalline \`a poids de Hodge--Tate $\{\ell+p,0\}$.

La partie (i) du sous-lemme se d\'emontre en prenant $\ell = k-(p+1)$, et la partie (ii) en prenant $\ell = k-2(p+1)$ et en
appliquant $\theta$ deux fois. Bien s\^ur, on pourrait appliquer $\theta$ plusieurs fois et obtenir des r\'esultats analogues
pour $k \ge 3p$, etc.

Remarquons qu'une preuve locale du sous-lemme s'ensuit de~\cite{Pa}, th.~8.6 (avant c'\'etait connu quand $k \le 2p$~\cite{B}, \S
3.2).  Une preuve globale, plus sophistiqu\'ee mais valable en toute g\'en\'eralit\'e, r\'esulte aussi du travail de
Kisin sur la conjecture de Breuil--M\'ezard, au moins si $p > 3$~\cite{Ki}, cor.~2.3.4.
\end{proof}

Dans certains cas (comme dans le travail de l'un des auteurs sur la conjecture de modularit{\'e} des surfaces
ab{\'e}liennes), il peut \^etre utile de formuler des conjectures sur l'existence de classes compagnons en cohomologie de de
Rham modulo~$p$ (et m{\^e}me de classes de de Rham classiques ou $p$-adiques relevant ces classes) sous des hypoth{\`e}ses de
d{\'e}composabilit{\'e} pour $\rhob\vert_{I_p}$ moins fortes que la d{\'e}composabilit{\'e} totale. Nous formulons de telles
conjectures ci-dessous pour chacun des huit poids compagnons.
  
\section{D{\'e}composabilit{\'e} partielle et formes compagnons}\label{sec:twists}

Soit $\rhob=\rhob_{f,p}:\Gamma\rightarrow \GSp_4(\kappa)$, suppos\'ee absolument irr\'eductible, la repr{\'e}sentation
r{\'e}siduelle associ{\'e}e {\`a} une forme cuspidale holomorphe $p$-ordinaire $f$, de poids $(k,\ell)$, $k\geq \ell\geq 3$.
En particulier, $\rhob$ est motiviquement impaire.  Dans tout ce qui suit, on fait l'hypoth{\`e}se que $k+\ell-3<p-1$. Soit
$k=a+3$ et $\ell=b+3$; la repr\'esentation $\rhob$ admet donc le poids de Serre $F(\lambda)=W(\lambda)$ o\`u
$\lambda=(a,b;a+b)\in C_0$.

Dans la formulation ci-dessous g\'en\'eralisant la th\'eorie des formes compagnons dans le contexte de la conjecture de Serre pour ${\GL_2}_{/\Q}$ (voir par exemple \cite{G}), on va introduire des twists par des puissances du caract{\`e}re
cyclotomique modulo~$p$ de la repr{\'e}sentation globale
$\overline{\rho}$ et formuler une conjecture d'existence de formes compagnons pour ces twists sous des hypoth\`eses de d\'ecomposabilit\'e partielle. 
On pourrait proc{\'e}der exactement de la m{\^e}me mani{\`e}re avec la repr{\'e}sentation $p$-adique elle-m{\^e}me et des twists par des puissances du caract{\`e}re cyclotomique $p$-adique sous des hypoth{\`e}ses analogues pour la repr{\'e}sentation $p$-adique $\rho_{f,p}\vert_{D_p}$, mais alors, les formes compagnons dont l'existence est {\'e}galement conjectur{\'e}e seraient $p$-adiques (voir \cite{BE} pour ${\GL_2}_{/\Q}$).

\vskip 3mm

\noindent{\bf Cas 1:}  Supposons qu'on ait
$$\overline{\rho}\vert_{I_p}\sim \left(\begin{array}{cccc}\omega^{k+\ell-3}&0&*&*\\0&\omega^{k-1}&*&*\\0&0&\omega^{\ell-2}&0
\\0&0&0&1\end{array}\right).$$

On forme dans ce cas $\overline{\rho}_1=\overline{\rho}\otimes\omega^{2-\ell}$, qui est toujours motiviquement impaire. On voit qu'apr{\`e}s conjugaison par $s_0=\iota(s_1)$, on a 
$$(s_0\cdot  \overline{\rho}_1\cdot s_0)|_{I_p}=\left(\begin{array}{cccc}\omega^{k-\ell+1}&0&*&*\\0&\omega^{k-1}&*&*\\0&0&\omega^{2-\ell}&0\\0&0&0&1\end{array}\right).$$

En posant $(k',\ell')=(k+p-1,4-\ell+p-1)$, on obtient un poids cohomologique: $k'\geq \ell'\geq 3$; nous conjecturons alors qu'il existe 
une repr{\'e}sentation cuspidale $\pi'$ de ce poids, avec $\pi'_\infty$ de type holomorphe de param{\`e}tre de Harish-Chandra $(k'-1,\ell'-2;k'+\ell'-6)$ et telle que
$$\overline{\rho}_{\pi',p}\cong \overline{\rho}\otimes\omega^{2-\ell}.$$
Le poids de Serre correspondant est $F(\lambda')$ pour $\lambda'=(a+p-1,p-3-b;a+b)$ qui est situ{\'e} dans l'alc\^ ove $C_3$ si $a-b > 1$.
Remarquons que l'existence de $\pi'$ n'implique pas tout \`a fait que $F(\lambda')\in\cW(\rhob)$, mais seulement (essentiellement) que
$\rhob^\vee\subset H^3_{et}(X\times\overline{\Q},V_{\lambda'}(\fpb))$. Rappelons que le syst\`eme local $V_{\lambda'}(\fpb)$ est r\'eductible en g\'en\'eral
puisque $\lambda'\not\in C_0$ (voir \eqref{eq:1}--\eqref{eq:10} ci-dessus).

Nous verrons dans la section suivante comment ceci peut {\^e}tre d{\'e}montr{\'e} en partie (les d{\'e}tails para\^ itront
dans \cite{T2}). Les deux ingr{\'e}dients-cl{\'e} sont: le th{\'e}or{\`e}me de comparaison {\'e}tale--cristallin modulo~$p$
de Faltings \cite{Fa}, le calcul de la cohomologie de de Rham et de sa filtration de Hodge {\`a} l'aide d'un complexe
filtr{\'e}, dit \og de Bernstein--Gelfand--Gelfand dual\fg\ introduit par Faltings--Chai \cite{FC}, qui est quasi-isomorphe
au complexe de de Rham logarithmique muni de sa filtration de Hodge. Ce Cas~1 est celui qui intervient dans l'{\'e}tude de la
conjecture de Yoshida (\cite{Y} et \cite{T1}).

\vskip 3mm

\noindent{\bf Cas 2:} Supposons qu'on ait
$$\overline{\rho}\vert_{I_p} \sim \left(\begin{array}{cccc}\omega^{k+\ell-3}&*&0&*\\0&\omega^{k-1}&0&0\\ 0&0&\omega^{\ell-2}&*
\\0&0&0&1\end{array}\right).$$

On forme dans ce cas $\overline{\rho}_2=\overline{\rho} \otimes  \omega^{1-k}$. On voit qu'apr{\`e}s conjugaison par $s_0 s_1=\iota(s_1 s_0)$, on a 
$$(s_0 s_1\cdot  \overline{\rho}_2\cdot s_1s_0)|_{I_p}=\left(\begin{array}{cccc}\omega^{\ell-k-1}&0&*&0\\0&\omega^{\ell-2}&*&*\\0&0&\omega^{1-k}&0\\0&0&0&1\end{array}\right).$$

En posant $(k',\ell')=(\ell-1+p-1,3-k+p-1)$, on obtient un poids cohomologique: $k'\geq \ell'\geq 3$; nous conjecturons alors qu'il existe 
une repr{\'e}sentation cuspidale $\pi'$ de ce poids, avec $\pi'_\infty$ de type Whittaker avec param{\`e}tre de Harish-Chandra $(k'-1,\ell'-2;k'+\ell'-6)$ et telle que
$$\overline{\rho}_{\pi',p}\cong \overline{\rho}\otimes\omega^{1-k}.$$
Le poids de Serre correspondant est $F(\lambda')$ pour $\lambda'=(b+p-2,p-4-a;a+b)$ qui est situ{\'e} dans l'alc\^ ove $C_2$ si $b > 0$.

On conjecture donc que la repr{\'e}sentation $\pi'$ intervient dans le terme gradu{\'e} de degr{\'e} $k'-1$ pour la filtration de Hodge de la cohomologie de deRham logarithmique (consid{\'e}r{\'e}e comme module de Hecke):
$$\gr^{k'-1}H^3_{ldR}(X/\overline{\Z}_p,\cV_{\lambda'})=H^1(\overline{X},\omega^{k',4-\ell'}).$$
Ici $\overline{X}$ d\'esigne une compactification toro\"idale arithm\'etique de $X$ sur $\Zp$ (voir chap.\ IV de \cite{FC}).

\vskip 3mm

\noindent{\bf Cas 3:} Supposons qu'on ait
$$\overline{\rho}\vert_{I_p}\sim \left(\begin{array}{cccc}\omega^{k+\ell-3}&0&0&0\\0&\omega^{k-1}&*&0\\0&0&\omega^{\ell-2}&0
\\0&0&0&1\end{array}\right).$$

On forme dans ce cas $\overline{\rho}_3=\overline{\rho}\otimes\omega^{3-k-\ell}$. On voit qu'apr{\`e}s conjugaison par 
$s_0 s_1 s_0=\iota(s_1 s_0 s_1)$, on a 
$$(s_0 s_1 s_0\cdot  \overline{\rho}_3\cdot s_0 s_1 s_0)|_{I_p}=\left(\begin{array}{cccc}\omega^{3-k-\ell}&0&0&0\\0&\omega^{2-\ell}&*&0\\0&0&\omega^{1-k}&0\\0&0&0&1\end{array}\right).$$

En posant $(k',\ell')=(3-\ell+p-1,3-k+p-1)$, on obtient un poids cohomologique: $k'\geq \ell'\geq 3$;  nous conjecturons alors qu'il existe 
une repr{\'e}sentation cuspidale $\pi'$ de ce poids, avec $\pi'_\infty$ de type Whittaker avec param{\`e}tre de Harish-Chandra $(k'-1,\ell'-2;k'+\ell'-6)$ et telle que
$$\overline{\rho}_{\pi',p}\cong \overline{\rho}\otimes\omega^{3-k-\ell}.$$
Le poids de Serre correspondant est $F(\lambda')$ pour $\lambda'=(p-4-b,p-4-a;a+b)$ qui est situ{\'e} dans l'alc\^ ove $C_1$ si $a+b < p-5$.

On conjecture donc que la repr{\'e}sentation $\pi'$ intervient dans
\[ \gr^{\ell'-2}H^3_{ldR}(X/\overline{\Z}_p,\cV_{\lambda'})=H^2(\overline{X},\omega^{\ell'-1,3-k'}). \]

\vskip 3mm
Il y a quatre autres cas, pour lesquels
la m\'ethode de d\'emonstration que nous proposons dans la section suivante  ne s'applique probablement pas. Des discussions avec E.\ Urban et L.\ Clozel sugg\`erent que les poids de Serre pr\'edits ci-dessous donnent lieu \`a des classes $c$ de cohomologie de de Rham modulo~$p$ obtenues par r\'eduction de formes compagnons $p$-adiques ordinaires.

Notons  que seuls les quatre repr{\'e}sentants de Kostant $w\in W^M$ ont la propri{\'e}t{\'e}
d'envoyer un poids $G$ dominant sur un poids $M$-dominant par $\lambda\mapsto w\cbull \lambda$; or les caract{\`e}res $M$-dominants $(c,d;e)$ de $T$ sont ceux pour lesquels il existe un entier $f$ tel que $(f+c,f+d;e)$ soit $G$-dominant.
Ces quatre autres cas sont:

\vskip 3mm
 
\noindent{\bf Cas 0$'$:} Si
$$\overline{\rho}|_{I_p}\sim \left(\begin{array}{cccc}\omega^{k+\ell-3}&*&*&*\\0&\omega^{k-1}&0&*\\0&0&\omega^{\ell-2}&*
\\0&0&0&1\end{array}\right),$$
la repr\'esentation $\overline{\rho}'_0=\overline{\rho}$ satisfait 
 $$(s_1 \cdot  \overline{\rho}'_0\cdot s_1)|_{I_p}=\left(\begin{array}{cccc}\omega^{k+\ell-3}&*&*&*\\0&\omega^{\ell-2}&0&*\\0&0&\omega^{k-1}&*\\0&0&0&1\end{array}\right).$$

On pose $(k',\ell')=(\ell-1+p-1,k+1)$ avec $k'\geq \ell'\geq 3$. Le poids de Serre correspondant est $F(\lambda')$ pour $\lambda' = (b+p-2,a+1;a+b+p-1)$ qui est situ\'e dans l'alc{\^o}ve $C_2$ si $b > 0$. 
 
 \vskip 3mm
 
 \noindent{\bf Cas 1$'$:} Si
$$\overline{\rho}|_{I_p}\sim \left(\begin{array}{cccc}\omega^{k+\ell-3}&0&*&0\\0&\omega^{k-1}&*&*\\0&0&\omega^{\ell-2}&0
\\0&0&0&1\end{array}\right),$$
la repr\'esentation  $\overline{\rho}'_1=\overline{\rho}\otimes \omega^{2-\ell}$ satisfait 
 $$(s_1 s_0\cdot  \overline{\rho}'_1\cdot s_0s_1)|_{I_p}=\left(\begin{array}{cccc}\omega^{k-\ell+1}&*&0&*\\0&\omega^{2-\ell}&0&0\\0&0&\omega^{k-1}&*\\0&0&0&1\end{array}\right).$$

On pose $(k',\ell')=(3-\ell+p-1,k+1)$ avec $k'\geq \ell'\geq 3$. Le poids de Serre correspondant est $F(\lambda')$ pour $\lambda' = (p-4-b,a+1;a+b+p-1)$ qui est situ\'e dans l'alc{\^o}ve $C_1$ si $a > b$.

 \vskip 3mm
 
 \noindent{\bf Cas 2$'$:} Si
$$\overline{\rho}|_{I_p}\sim \left(\begin{array}{cccc}\omega^{k+\ell-3}&*&0&0\\0&\omega^{k-1}&0&0\\0&0&\omega^{\ell-2}&*
\\0&0&0&1\end{array}\right),$$
la repr\'esentation $\overline{\rho}'_2=\overline{\rho}\otimes\omega^{1-k}$ satisfait 
 $$(s_1 s_0 s_1 \cdot  \overline{\rho}'_2\cdot s_1 s_0 s_1)|_{I_p}=\left(\begin{array}{cccc}\omega^{\ell-k-1}&*&0&0\\0&\omega^{1-k}&0&0\\0&0&\omega^{\ell-2}&*\\0&0&0&1\end{array}\right).$$

On pose $(k',\ell')=(2-k+p-1,\ell)$ avec $k'\geq \ell'\geq 3$. Le poids de Serre correspondant est $F(\lambda')$ pour $\lambda' = (p-5-a,b;a+b+p-1)$ qui est situ\'e dans l'alc{\^o}ve $C_0$. 
 
 \vskip 3mm
 
 \noindent{\bf Cas 3$'$:} Si
$$\overline{\rho}|_{I_p}\sim \left(\begin{array}{cccc}\omega^{k+\ell-3}&0&0&0\\0&\omega^{k-1}&0&0\\0&0&\omega^{\ell-2}&0
\\0&0&0&1\end{array}\right),$$
la repr\'esentation $\overline{\rho}'_3=\overline{\rho}\otimes\omega^{3-k-\ell}$ satisfait 
 $$(s_0 s_1 s_0 s_1 \cdot  \overline{\rho}'_3\cdot s_1 s_0 s_1 s_0)|_{I_p}=\left(\begin{array}{cccc}\omega^{3-k-\ell}&0&0&0\\0&\omega^{1-k}&0&0\\0&0&\omega^{2-\ell}&0\\0&0&0&1\end{array}\right).$$

On pose $(k',\ell')=(2-k+2(p-1),4-l+p-1)$ avec $k'\geq \ell'\geq 3$. Le poids de Serre correspondant est $F(\lambda')$ pour $\lambda' = (2p-6-a,p-3-b;a+b+p-1)$ qui est situ\'e dans l'alc{\^o}ve $C_3$ si $a+b < p-6$.

\section{Complexe BGG dual coh{\'e}rent}\label{sec:bgg}

Dans cette section on esquisse une strat\'egie pour construire une forme compagnon dans le Cas~1 qui utilise le complexe BGG
dual. Pour l'instant cette strat\'egie ne fournit qu'une forme \emph{$p$-adique} du poids souhait\'e. Les d\'etails se
trouvent dans un travail du deuxi\`eme auteur \cite{T2}.

Posons $k=a+3$ et $\ell=b+3$, $a\geq b\geq 0$ avec $a+b+3<p-1$ et $\lambda=(a,b;a+b)$ le poids de $G$ associ{\'e}.
Soit $X$ une vari{\'e}t{\'e} de Shimura pour $G$, de niveau $N$ premier {\`a}~$p$ et net.
Soit $\cV_\lambda$ le fibr{\'e} {\`a} connexion associ{\'e} {\`a} la $\Z_{(p)}$-repr{\'e}sentation de $G$ de plus haut poids $\lambda$ 
(voir \cite{PT}, 1.9 pour l'unicit{\'e} d'un tel mod{\`e}le entier sur $\Z_{(p)}$ lorsque $\lambda \in C_0$, ce qui est le cas ici). 
Soit $\pi$ une repr{\'e}sentation cuspidale cohomologique de $G(\A)$ intervenant dans $H^3_{ldR}(X,\cV_{\lambda})\otimes_{\Z_{(p)}} \C$; $\pi$ d{\'e}finit
 alors une forme modulaire de poids $(k,\ell)$
soit holomorphe soit admettant un mod{\`e}le de Whittaker; notons que $\pi$ poss\`ede des vecteurs non nuls par le
sous-groupe de congruences principal de niveau $N$.
On fixe un plongement $\overline{\Q}\rightarrow \qpb$ et un anneau de valuation discr{\`e}te $\cO\subset \overline{\Q}_p$ fini sur $\zp$ et
 contenant les valeurs propres de $\pi$ pour les op{\'e}rateurs de Hecke $T_{q,i}$ pour tous les nombres premiers $q$ ne divisant pas le niveau $N$ (on inclut donc $q=p$). On notera $\varpi$ une uniformisante de $\cO$ et $\kappa$ son corps r\'esiduel.

On remplace $X$ et $\cV_\lambda$ par ses changements de base \`a $\zp$.
Soit $\overline{X}$ une compactification toro\"idale lisse de $X$ sur $\zp$. Le complexe BGG dual pour $G$ est un complexe filtr{\'e} 
$\cK^\expbull_\lambda$ contenu comme facteur direct dans le complexe de de Rham logarithmique $\cV_{\lambda}\otimes\Omega_{\overline{X}}^\expbull(\log)$ 
sur $\overline{X}$ (filtr\'e par la filtration convol\'ee de la filtration b\^ete de $\Omega_{\overline{X}}^\expbull(\log)$ et de la filtration de Hodge de $\cV_\lambda$). Il est constitu{\'e}
de faisceaux localement libres de rang fini. Nous renvoyons {\`a} \cite{PT} et \cite{MT} pour une description d{\'e}taill{\'e}e de sa structure enti{\`e}re sur $\zp$. L'inclusion  $\cK^\expbull_{\lambda}\hookrightarrow \cV_{\lambda}\otimes\Omega_{\overline{X}}^\expbull(\log)$
est un quasi-isomorphisme filtr{\'e} \cite{MT}, sect.~5, th.~6. Pour tout triplet d'entiers $(u,v,t)$ avec $u\geq v$, consid\'erons le faisceau $\omega^{u,v}(t)$ sur $\overline{X}$ des sections du fibr\'e automorphe associ\'e \`a la repr\'esentation ${\rm Sym}^{u-v}\otimes {\rm det}^v\Z_{p}^2(t)$ du sous-groupe de Levi $\GL_2\times \GL_1$ du parabolique de Siegel, l'action du facteur $\GL_1$ \'etant donn\'ee par la puissance $t$-i\`eme du facteur de similitudes.
On peut alors {\'e}crire les termes du complexe $\cK^\expbull_{\lambda}$; ils sont plac{\'e}s en degr{\'e}
$w$, $w+1$, $w+2$ et $w+3$, o\`u $w=a+b+3$ (le poids motivique de $\pi$): 
$$\omega^{3-\ell,3-k}(k+\ell-6)\rightarrow \omega^{\ell-1,3-k}(k-4)\rightarrow \omega^{k,4-\ell}(\ell-5) \rightarrow \omega^{k,\ell}(-3),$$
les morphismes {\'e}tant des op{\'e}rateurs diff{\'e}rentiels induits par la connexion de Gauss--Manin.
On sait que la connexion de Gauss--Manin commute aux correspondances alg\'ebriques: les diff\'erentielles ci-dessus commutent donc aux op\'erateurs de Hecke.
Cependant si l'on omet les twists par des puissances du facteur de similitudes, il faut les r\'eintroduire dans les formules de commutation.

\begin{rk}
  Les twists n'interviennent pas pour les consid\'erations g\'eom\'etriques qui suivent, mais sont cruciaux pour la
  description des valeurs propres de Hecke sur la forme compagnon qu'on veut construire.
\end{rk}

La filtration de Hodge sur le complexe de de Rham induit sur le complexe BGG dual la filtration b\^ete. On en d\'eduit les gradu\'es  de $H^3_{ldR}(X,\cV_{\lambda})$:
\begin{equation}\label{eq:gr} 
\begin{split}
\gr^0 H^3_{ldR}=H^3(\overline{X},\omega^{3-\ell,3-k}),&\quad 
\gr^{\ell-2}H^3_{ldR}=H^2(\overline{X},\omega^{\ell-1,3-k}), \\
\gr^{k-1}H^3_{ldR}=H^1(\overline{X},\omega^{k,4-\ell}),&\quad
\gr^{k+\ell-3}H^3_{ldR}=H^0(\overline{X},\omega^{k,\ell}).
\end{split}
\end{equation}

Les hypoth{\`e}ses de d{\'e}composabilit{\'e} partielle des Cas 1, 2, 3 de la section pr{\'e}c{\'e}dente ont une traduction
via le th{\'e}or{\`e}me de comparaison {\'e}tale--cristallin de Faltings \cite{Fa}, en termes de stabilit{\'e} par le
Frobenius cristallin $\Phi$ de certains sous-quotients de la filtration de Hodge sur le $\Phi$-module filtr{\'e} $\overline{M}_\pi$
associ{\'e} {\`a} la repr{\'e}sentation contragr\'ediente de $\overline{\rho}_{\pi,p}$. Par la th{\'e}orie de
Fontaine--Laffaille, on a $\overline{M}_\pi=M_\pi\otimes_\cO \kappa$ o{\`u} $M_\pi$ est le $\Phi$-module filtr{\'e}, libre de
rang~4 sur $\cO$, associ{\'e} {\`a} la repr{\'e}sentation contragr\'ediente de $\rho_{\pi,p}$; par \cite{MT}, $M_\pi$ est
facteur direct du $\cO$-module $H^3_{cris}(\overline{X}_{/\F_p},\cV_\lambda) \otimes \cO$ muni du Frobenius cristallin $\Phi$
et de la filtration convol\'ee d\'ecrite ci-dessus. Par ordinarit\'e de $\rho_{\pi,p}$, on a une suite exacte courte de
$\Phi$-modules filtr\'es
$$0\rightarrow M'_\pi\rightarrow M_\pi\rightarrow M''_\pi\rightarrow 0
$$
o\`u $M'_\pi$ et $M''_\pi$ sont libres de rang $2$ sur $\cO$; et la surjection $M_\pi\rightarrow M''_\pi$ induit ${\rm Fil}^{k-1}M_\pi\cong M''_\pi$ et ${\rm Fil}^{k+\ell-3}M_\pi\cong {\rm Fil}^{k+\ell-3}M''_\pi$.  
Soit $P_{f,p}(X)=(X-\alpha)(X-\beta)(X-\gamma)(X-\delta)$ le polyn\^ome de Hecke en $p$. Ses racines sont ordonn\'ees pour que
$\alpha, {\beta\over p^{\ell-2}},{\gamma\over p^{k-1}},{\delta\over p^{k+\ell-3}}$ soient des unit\'es $p$-adiques. Par \cite{U}, les valeurs propres de $\Phi$
sur $M_\pi$, resp. sur $M''_\pi$, sont $\alpha,\beta,\gamma,\delta$, resp. $\gamma,\delta$. 
\vskip 3mm

Dans le Cas~1, on suppose que $\pi$ est donn{\'e}e par une forme cuspidale holomorphe $f$ ordinaire en $p$; pour tout premier
$q$ ne divisant pas $Np$, on note $a_{q,i}$ les valeurs propres des op\'erateurs de Hecke $T_{q,i}$ ($i=1,2$). On
consid{\`e}re la forme diff{\'e}rentielle $\omega_f\in H^0(\overline{X},\omega^{k,\ell})\otimes\cO$ associ\'ee \`a $f$. Cette
forme d{\'e}finit une classe dans ${\rm Fil}^{k+\ell-3}M_\pi$.  Soit $[\overline{\omega}_f]\in {\rm
  Fil}^{k+\ell-3}\overline{M}_\pi$ la classe de cohomologie de sa r\'eduction modulo $\varpi$. La traduction de l'hypoth\`ese
de scindage partiel du Cas~1 en termes de la th\'eorie de Fontaine--Laffaille est que le $\Phi$-module filtr\'e
$\overline{M}''_\pi$ est scind\'e. C'est-\`a-dire que ${\rm Fil}^{k+\ell-3} \overline M''_\pi$, le dernier cran non nul de la
filtration de Hodge de ce module, est stable par le Frobenius.
C'est donc que ${\Phi-\delta\over p^{k+\ell-3}}$ annule ${\rm Fil}^{k+\ell-3} \overline M''_\pi$
(un argument similaire \`a celui de
\cite{FJ} et \cite{BE}), et on en d\'eduit

\begin{lem}[\cite{T2}, prop.\ 9.1]\label{lm:image-nulle} L'image par ${\Phi-\delta\over p^{k+\ell-3}}$ de $[\overline{\omega}_f]$ est d'image nulle
  dans $\overline{M}''_\pi$.
\end{lem}

 Soit $V$, resp. $\overline{V}$ le lieu ordinaire de $X$, resp. $\overline{X}$; c'est le $\Z_p$-sous-sch{\'e}ma formel ouvert du
compl{\'e}t{\'e} de $\overline{X}$ le long de la fibre sp{\'e}ciale, d{\'e}fini par la non-annulation modulo~$p$ de
l'invariant de Hasse $H$. L'endomorphisme de Frobenius associ\'e au sous-groupe canonique agit sur le $\Z_p$-sch{\'e}ma formel $\overline{V}$. Son action sur la cohomologie log de Rham $H^\expbull_{ldR}(V,\cV_\lambda)$
est compatible avec celle du Frobenius cristallin $\Phi$.

On peut donc
faire agir ${\Phi-\delta\over p^{k+\ell-3}}$ sur la restriction $\omega_f\vert_{\overline V}$ de $\omega_f\in H^0(\overline{X},\omega^{k,\ell})\otimes \cO$ {\`a} $\overline{V}$;
un calcul de $q$-d\'eveloppement facile montre  

\begin{lem}[\cite{T2}, lemme 10.6] La forme $\xi_f=({\Phi-\delta\over p^{k+\ell-3}})\omega_f\vert_{\overline V}\in
  H^0(\overline{V},\omega^{k,\ell})\otimes \cO$ est non-nulle modulo~$\varpi$.
\end{lem}

Soit $X^*$ la compactification minimale de $X$ et notons $\pi : \overline X \to X^*$ le morphisme canonique. Soit
$\overline{V}' := \overline V - \pi^{-1}(X_0)$ o\`u $X_0 \subset X^*$ et l'ensemble fini de pointes.
Soit $\cH^{Np}$ l'alg\`ebre de Hecke abstraite hors de $Np$ sur $\cO$ et $\m$ l'id\'eal maximal engendr\'e par $\varpi$ et
les $T_{q,i}-a_{q,i}$. On sait que la diff\'erentielle $d:\omega^{k,4-\ell}(\ell-5)\rightarrow \omega^{k,\ell}(-3)$ commute
\`a l'action des correspondances de Hecke. La partie technique de la d\'emonstration du th\'eor\`eme principal consiste alors
\`a \'etablir la proposition suivante.

\begin{pro}[\cite{T2}, th.\ 3] La localisation en $\m$ de la suite spectrale donn\'ee par la filtration b\^ete sur le
  complexe BGG dual induit une suite exacte courte
$$0\rightarrow H^0(\overline{V}',\omega^{k,4-\ell}(\ell-5))_\m\rightarrow H^0(\overline{V}',\omega^{k,\ell}(-3))_\m\rightarrow
H^{w+3}(\overline{V}',\cK^\expbull_\lambda)_\m\rightarrow 0$$
dont la r\'eduction modulo $p$ reste exacte. De plus, le morphisme
$$M_\pi \subset H^3_{ldR}(X,\cV_{\lambda})_\m = H^{w+3}(\overline X, \cK^\expbull_\lambda)_\m
\rightarrow H^{w+3}(\overline{V}',\cK^\expbull_\lambda)_\m$$
se factorise par la projection $M_\pi\rightarrow M''_\pi$.
\end{pro}

Notons que $H^0(\overline{V}',\omega^{i,j}) = H^0(\overline{V},\omega^{i,j}) = H^0({V},\omega^{i,j})$ par le principe
de Koecher.

Consid\'erons alors l'id\'eal maximal $\widetilde{\m}$ de $\cH^{Np}$ engendr\'e par $\varpi$ et
les $T_{q,i}-q^{-i(\ell-2)}a_{q,i}$.
Par d\'efinition du twist par $\ell-2$, on voit que $H^0(\overline{V},\omega^{k,4-\ell}(\ell-5))_\m=H^0(\overline{V},\omega^{k,4-\ell}(-3))_{\widetilde{\m}}(\ell-2)$.

Il r\'esulte ais\'ement de la proposition et du lemme~\ref{lm:image-nulle} que la r{\'e}duction $\overline{\xi}_f$ de $\xi_f$
modulo~$\varpi$ est dans l'image du morphisme $\omega^{k,4-\ell}(\overline{V})_{\widetilde{\m}}(\ell-2)\otimes
\kappa\rightarrow \omega^{k,\ell}(\overline{V})_\m\otimes\kappa$ induit par la diff\'erentielle $d:\cK^{w+2}_\lambda\rightarrow
\cK_\lambda^{w+3}$ du complexe BGG dual.

Ceci montre qu'il existe  
une forme $\overline{\omega}_{g_1}$ ant\'ec\'edente de $\overline{\xi}_f$ par $d^2$, vecteur propre g\'en\'eralis\'e pour les
op{\'e}rateurs de Hecke hors de $Np$ et pour $U_{p,2}$ (on ne sait cependant pas que $\overline{g}_1$ est
vecteur propre g\'eneralis\'e pour $U_{p,1}$). La forme $\overline{g}_2 = \overline{g}_1 H$ obtenue par multiplication par
l'invariant de Hasse $H$
est de poids $(k',\ell')$ pour $k'=k+(p-1) \ge \ell'=4-\ell+(p-1) \ge 3$ est \'el\'ement de l'espace propre g\'en\'eralis\'e
$H^0(V,\omega^{k',\ell'})_{\widetilde{\m}}\otimes \kappa$.
Comme cet espace est r\'eunion croissante d'espaces de dimension finie stables
par les op\'erateurs de Hecke, il contient une forme propre $\overline{g}_3$. Soit $]\overline V[$ le tube rigide de l'ouvert
$\overline V$ de la fibre sp\'eciale de $\overline X$.
Par une variante simple du Going-Down Lemma de Cohen--Seidenberg, le syst\`eme de valeurs propres de $\overline g_3$
se rel\`eve en un syst\`eme de valeurs propres en caract\'eristique z\'ero, qui d\'efinit une forme propre $p$-adique dans
$H^0(]\overline V[,\omega^{k',\ell'})\otimes \cO$
qui a les propri\'et\'es voulues: elle est propre pour les $T_{q,i}$ ($q$ premier \`a $p$) et
pour $U_{p,2}$ et ses valeurs propres $\lambda_{q,i}$ sont congrues modulo $\varpi$ aux valeurs propres
$q^{-i(\ell-2)}a_{q,i}$ pour $q$ premier \`a $Np$, et $\lambda_{p,2}$ est congrue \`a ${\alpha\beta\over p^{\ell-2}}$ (elle
est en particulier ordinaire pour $U_{p,2}$).  Voir \cite{T2}, sect.~10.3 pour les d\'etails.
Cette forme $p$-adique est conjecturalement ordinaire pour $U_{p,1}$. Il est facile de
voir, par exemple par la th\'eorie de Hida, qu'elle serait alors classique.
L'\'enonc\'e contenu dans \cite{T2} (voir section~5) est le suivant

\begin{thm}
  Soit $N$ un entier sans facteurs carr{\'e}s et soient $k$, $\ell$ deux entiers tels que $k\geq \ell\geq 3$. Soit $f$ une
  forme cuspidale holomorphe de poids $(k,\ell)$ de niveau $N$, propre pour les op{\'e}rateurs de Hecke $T_{q,1}, T_{q,2}$
  \($q$ premier {\`a} $N$\). Soit $p$ premier tel que $p-1>k+\ell-3$, $p$ ne divise pas $N$, et tel que $f$ est ordinaire
  en~$p$ \(voir d\'ef.~\ref{df:ord}\). Soit $\rho_{f,p}:{\rm Gal}(\overline{\Q}/\Q)\rightarrow \GSp_4(\cO)$ la
  repr{\'e}sentation galoisienne $p$-adique associ{\'e}e et $\overline{\rho}=\overline{\rho}_{f,p}$ sa r{\'e}duction modulo
  l'id{\'e}al maximal $(\varpi)$ de $\cO$. Supposons que
  \begin{enumerate}
  \item l'image de $\overline{\rho}$ contient $\Sp_4(\kappa')$ pour un corps fini $\kappa'$,
  \item $\overline{\rho}$ est minimalement ramifi{\'e}e en tout premier divisant $N$ \(cf.\ \cite{GT}, sect.~3\),
  \item les nombres $\alpha$, ${\beta\over p^{\ell-2}}$, ${\gamma\over p^{k-1}}$, ${\delta\over p^{k+\ell-3}}$ sont deux \`a
    deux distincts dans le corps r\'esiduel,
  \item on est dans le Cas~1 pour la restriction de $\overline{\rho}$ {\`a} $I_p$.
  \end{enumerate}
  Alors, il existe une forme $g$ cuspidale $p$-adique de poids $(k',\ell')=(k+(p-1),4-\ell+(p-1))$ et de niveau $N$, propre
  pour les op{\'e}rateurs de Hecke $T_{q,1}, T_{q,2}$ \($q$ premier {\`a} $pN$\), $T_{p,2}$-ordinaire, telle que
  $$\overline{\rho}_{g,p}=\overline{\rho}_{f,p}\otimes\omega^{2-\ell}.$$
\end{thm}

L'existence d'une forme compagnon classique dans le Cas~1 ainsi que pour les six autres syst\`emes de valeurs propres de
Hecke (c'est-\`a-dire pour les huit premiers poids de Serre), a \'et\'e \'etablie r\'ecemment par Gee--Geraghty \cite{GG}
sous des hypoth\`eses faibles, par un th\'eor\`eme de modularit\'e reposant sur des r\'esultats nouveaux de rel\`evement
modulaire.

\vskip 5mm

\noindent{\small F.\ Herzig, Institute for Advanced Study, Einstein Drive, Princeton, NJ 08540.
  {\'E}tats-Unis.
 
 \noindent {\tt herzig@math.ias.edu}}

\medskip
 
\noindent{\small J.\ Tilouine,   D{\'e}partement de Math{\'e}matiques, UMR 7539,
 Institut Galil{\'e}e, Universit{\'e} de Paris 13, 93430 Villetaneuse. France.
 
 \noindent {\tt tilouine@math.univ-paris13.fr}}

\end{document}

%% file: 20poids.pstex_t
\begin{picture}(0,0)%
\epsfig{file=20poids.pstex}%
\end{picture}%
\setlength{\unitlength}{2565sp}%
\begingroup\makeatletter\ifx\SetFigFont\undefined%
\gdef\SetFigFont#1#2#3#4#5{%
  \reset@font\fontsize{#1}{#2pt}%
  \fontfamily{#3}\fontseries{#4}\fontshape{#5}%
  \selectfont}%
\fi\endgroup%
\begin{picture}(3730,7036)(4576,-6635)
\put(4576,-6136){\makebox(0,0)[lb]{\smash{{\SetFigFont{8}{9.6}{\familydefault}{\mddefault}{\updefault}{\color[rgb]{0,0,0}$\alpha_1$}%
}}}}
\put(4801,-6586){\makebox(0,0)[lb]{\smash{{\SetFigFont{8}{9.6}{\familydefault}{\mddefault}{\updefault}{\color[rgb]{0,0,0}$-\widetilde\rho$}%
}}}}
\put(5401,-6436){\makebox(0,0)[lb]{\smash{{\SetFigFont{8}{9.6}{\familydefault}{\mddefault}{\updefault}{\color[rgb]{0,0,0}$\alpha_0$}%
}}}}
\put(7876,-3736){\makebox(0,0)[lb]{\smash{{\SetFigFont{8}{9.6}{\familydefault}{\mddefault}{\updefault}{\color[rgb]{0,0,0}$x-y$}%
}}}}
\put(5026,-1186){\makebox(0,0)[lb]{\smash{{\SetFigFont{8}{9.6}{\familydefault}{\mddefault}{\updefault}{\color[rgb]{0,0,0}$y$}%
}}}}
\put(4801,-5311){\makebox(0,0)[lb]{\smash{{\SetFigFont{7}{8.4}{\familydefault}{\mddefault}{\updefault}{\color[rgb]{0,0,0}$C_0$}%
}}}}
\put(4801,-3061){\makebox(0,0)[lb]{\smash{{\SetFigFont{7}{8.4}{\familydefault}{\mddefault}{\updefault}{\color[rgb]{0,0,0}$C_1$}%
}}}}
\put(7426,-3061){\makebox(0,0)[lb]{\smash{{\SetFigFont{7}{8.4}{\familydefault}{\mddefault}{\updefault}{\color[rgb]{0,0,0}$C_2$}%
}}}}
\put(7426,-811){\makebox(0,0)[lb]{\smash{{\SetFigFont{7}{8.4}{\familydefault}{\mddefault}{\updefault}{\color[rgb]{0,0,0}$C_3$}%
}}}}
\end{picture}%